\documentclass[a4paper,10pt]{amsart}
\usepackage{enumerate, amsmath, amsfonts, amssymb, amsthm, mathtools, thmtools, wasysym, graphics, graphicx, frcursive,xparse,comment,ytableau,multirow,comment,multicol}
\usepackage[normalem]{ulem}
\usepackage[boxed,norelsize]{algorithm2e}
\usepackage[all]{xy}
\usepackage{hyperref}
\usepackage{url, hypcap}
\hypersetup{colorlinks=true, citecolor=darkblue, linkcolor=darkblue}
\usepackage{tikz}
\usetikzlibrary{calc,through,backgrounds,shapes,matrix}
\usepackage{colortbl}
\definecolor{darkblue}{rgb}{0.0,0,0.7} 
\newcommand{\darkblue}{\color{darkblue}} 
\definecolor{darkred}{rgb}{0.7,0,0} 
\definecolor{lightgrey}{rgb}{0.7,0.7,0.7} 
\newtheorem{theorem}{Theorem}[section]
\newtheorem{proposition}[theorem]{Proposition}
\newtheorem{corollary}[theorem]{Corollary}
\newtheorem{lemma}[theorem]{Lemma}
\theoremstyle{definition}
\newtheorem{definition}[theorem]{Definition}
\newtheorem{example}[theorem]{Example}

\newtheorem{remark}[theorem]{Remark}
\usepackage[procnames]{listings}
\usepackage{cleveref}
\usepackage[all]{xy}
\usepackage[T1]{fontenc}
\crefformat{footnote}{#2\footnotemark[#1]#3}
\crefformat{conjecture}{Conjecture~#2#1#3}
\newcommand{\defn}[1]{\emph{\darkblue #1}}
\newcommand{\w}{{\sf w}}
\newcommand{\sw}{{\sf sweep}_m}

\newcommand{\swp}{{\sf presweep}_m}
\newcommand{\iswp}{{\sf inverse\_presweep}_m}
\newcommand{\f}{{\sf forget}}
\newcommand{\invf}{{\sf rightmost}}
\newcommand{\equ}{{\sf Equitable}}
\newcommand{\invfs}{{\sf successful}}
\newcommand{\uu}{{\sf u}}
\newcommand{\vv}{{\sf v}}
\newcommand{\us}{{\sf s}}
\newcommand{\ut}{{\sf t}}
\newcommand{\bb}{{\sf u}^*}

\newcommand{\block}{{\sf block}}
\newcommand{\swg}{{\sf sweep}}
\newcommand{\bmone}{{\blacksquare}}

\newcommand{\bmoner}{{\textcolor{red}{\blacksquare}}}
\newcommand{\bmzero}{{\cdot}}
\newcommand{\rightmost}{{\sf leftmost}}
\newcommand{\AZ}{{\mathcal{A}_{\mathbb{Z}}}}
\newcommand{\AN}{{\mathcal{A}_{\mathbb{N}}}}
\newcommand{\Am}{{\mathcal{A}}}
\newcommand{\DD}{{\mathcal{D}}}
\newcommand{\pat}{{\mathsf{path}}}
\renewcommand{\mod}{\operatorname{mod}}
\definecolor{keywords}{RGB}{255,0,90}
\definecolor{comments}{RGB}{0,0,113}
\definecolor{red}{RGB}{160,0,0}
\definecolor{green}{RGB}{0,150,0}
\crefname{algocf}{alg.}{algs.}
\Crefname{algocf}{Algorithm}{Algorithms}

\usepackage[colorinlistoftodos]{todonotes}

\title[Sweeping up Zeta]
  {Sweeping up Zeta}

\author[H.~Thomas]{Hugh Thomas}
\address[H.~Thomas]{LaCIM, Universit\'e du Qu\'ebec \`a Mont\'eal \\
Montr\'eal (Qu\'ebec), Canada}
\email{hugh.ross.thomas@gmail.com}

\author[N.~Williams]{Nathan Williams}
\address[N.~Williams]{LaCIM, Universit\'e du Qu\'ebec \`a Mont\'eal \\
Montr\'eal (Qu\'ebec), Canada}
\email{nathan.f.williams@gmail.com}

\date{\today}

\begin{document}

\begin{abstract}
We repurpose the main theorem of~\cite{thomas2014cyclic} to prove that modular sweep maps are bijective.  We construct the inverse of the modular sweep map by passing through an intermediary set of equitable partitions; motivated by an analogy to stable marriages, we prove that the set of equitable partitions for a fixed word forms a distributive lattice when ordered componentwise.  We conclude that the general sweep maps defined in~\cite{armstrong2014sweep} are bijective.  As a special case of particular interest, this gives the first proof that the zeta map on rational Dyck paths is a bijection.
\end{abstract}

\maketitle

\section{Introduction}
\label{sec:intro}

The sweep map of~\cite{armstrong2014sweep} is a broad generalization of the zeta map on Dyck paths, originally defined by J.~Haglund and M.~Haiman in the context of the study of diagonal harmonics.  Proving bijectivity of the sweep map was an open problem with significant implications in the study of rational Catalan combinatorics (see \Cref{sec:history}).  We solve this problem in~\Cref{thm:sweep_bij}.

To put the reader in the right frame of mind for our solution, we begin with a ``real world'' interpretation of the problem.

\subsection{Scheduling Tasks}
\label{sec:real_world}
Consider the following ``real world'' problem of scheduling recurrent daily tasks (for example, on a computer).  The day is divided into $m$ hours, and there are $N$ tasks to be carried out each day (numbered, or prioritized, from $1$ to $N$).  For simplicity, let us suppose that the total amount of time required to carry out all of the tasks is some multiple of $m$.  Each task takes an integer number of hours to complete: a task can take zero hours, but no task takes as much as an entire day.    Tasks start on the hour, and cannot be interrupted once started.  Tasks can be worked on concurrently, and the starting order for the tasks within the day is specified---the $j$th task must start before or at the same time as the $(j+1)$st task.  The schedule repeats every day, so a recurrent task can start at the end of one day and finish at the beginning of the next day.

Under these assumptions, it is reasonable to ask for an assignment of starting hours for the tasks so that the workload throughout the day is constant.  We call an answer---the starting times of the tasks---an \defn{equitable partition} for our scheduling problem.  It is not immediately clear that an answer necessarily exists.

\begin{example}
\label{ex:ex1}
Suppose we divide the day into $m=5$ hours, and we have a list of $N=7$ recurrent daily tasks such that the length (in hours) that it takes to complete the $j$th task is the $j$th element of the sequence $[1,3,3,1,4,2,1]$.  Since completion of all the tasks, one after the other, would take $(1+3+3+1+4+2+1) \text{ hours} = 15 \text{ hours}=3 \cdot 5 \text{ hours} = 3 \text{ days},$ we should work on three tasks at a time.

Starting tasks in the hours $[1,1,2,2,3,4,5]$ (so that tasks one and two both start in hour 1, tasks three and four in hour 2, and so on until task seven starts in hour 5) or in the hours $[1,2,2,4,5,5,5]$ both yield equitable partitions.  These are illustrated below---each task is assigned a row and each hour a column, and we mark which tasks we worked on in each hour with a symbol $\bmone$ in the corresponding cell.  The conditions for an equitable partition force each column to have the same number of copies of $\bmone$, and constrain the starting point of the $j$th row (colored red) to be weakly to the left of that of the $(j+1)$st row.

\[\begin{array}{cc|ccccc}
& & \multicolumn{5}{c}{{\text{Hours}}} \\
& & 1 & 2& 3& 4& 5\\ \hline
\multirow{7}{*}{{Tasks}}
&1& \bmoner & \bmzero & \bmzero & \bmzero & \bmzero \\
&2&\bmoner & \bmone & \bmone & \bmzero & \bmzero \\
&3&\bmzero & \bmoner & \bmone & \bmone & \bmzero\\
&4&\bmzero & \bmoner & \bmzero & \bmzero & \bmzero \\
&5&\bmone & \bmzero & \bmoner & \bmone & \bmone \\
&6&\bmzero & \bmzero & \bmzero & \bmoner & \bmone \\
&7&\bmzero & \bmzero & \bmzero & \bmzero & \bmoner
\end{array} \hspace{2em}\text{and}\hspace{2em} \begin{array}{cc|ccccc}
& & \multicolumn{5}{c}{{\text{Hours}}} \\
& & 1 & 2& 3& 4& 5\\ \hline
\multirow{7}{*}{{Tasks}}
&1& \bmoner & \bmzero & \bmzero & \bmzero & \bmzero \\
&2&\bmzero & \bmoner & \bmone & \bmone & \bmzero \\
&3&\bmzero & \bmoner & \bmone & \bmone & \bmzero\\
&4&\bmzero & \bmzero & \bmzero & \bmoner & \bmzero \\
&5&\bmone & \bmone & \bmone & \bmzero & \bmoner \\
&6&\bmone & \bmzero & \bmzero & \bmzero & \bmoner \\
&7&\bmzero & \bmzero & \bmzero & \bmzero & \bmoner
\end{array}\]
\end{example}


Now imagine that an (admittedly, somewhat fussy) inspector arrives at the last hour of a day, and wants to watch each of the tasks being done---one at a time, each one performed from beginning to end without interruption.\footnote{We must insist that the inspector arrives at the last hour, since for any other specified hour, one can construct a scheduling problem that has no equitable partitions with a task starting in that hour.}  After you complete a task under observation, the inspector next watches the lowest-numbered task beginning as promptly as possible among those tasks s/he has not yet watched.  A \defn{successful partition} is an equitable partition that allows our inspector to watch each task performed exactly once, without any delay between the end of one task and the beginning of another.

\begin{example}
\label{ex:ex2}
The equitable partition with starting hours $[1,2,2,4,5,5,5]$ (on the right in~\Cref{ex:ex1}) is successful because our inspector can observe the tasks without interruption as follows:

\[\begin{array}{cc|c|ccccc|ccccc|cccc}
& & \multicolumn{15}{c}{{\text{Hours}}} \\
& & 5 & 1& 2& 3& 4& 5 & 1& 2& 3& 4& 5 & 1& 2&3& 4\\ \hline
\multirow{7}{*}{{Tasks}}
&1 & \bmzero&\bmzero & \bmzero & \bmzero & \bmzero & \bmzero&\bmzero & \bmzero & \bmzero & \bmzero & \bmzero & \bmone & \bmzero & \bmzero & \bmzero\\
&2 & \bmzero&\bmzero & \bmzero & \bmzero & \bmzero & \bmzero&\bmzero & \bmone & \bmone & \bmone & \bmzero &\bmzero & \bmzero & \bmzero & \bmzero\\
&3 & \bmzero&\bmzero & \bmzero & \bmzero & \bmzero & \bmzero&\bmzero & \bmzero & \bmzero & \bmzero & \bmzero&\bmzero & \bmone & \bmone & \bmone\\
&4 & \bmzero&\bmzero & \bmzero & \bmzero & \bmone & \bmzero&\bmzero & \bmzero & \bmzero & \bmzero & \bmzero&\bmzero & \bmzero & \bmzero & \bmzero \\
&5 & \bmone&\bmone & \bmone & \bmone & \bmzero & \bmzero&\bmzero & \bmzero & \bmzero & \bmzero & \bmzero &\bmzero & \bmzero & \bmzero & \bmzero\\
&6 & \bmzero&\bmzero & \bmzero & \bmzero & \bmzero & \bmone&\bmone & \bmzero & \bmzero & \bmzero & \bmzero &\bmzero & \bmzero & \bmzero & \bmzero\\
&7 & \bmzero&\bmzero & \bmzero & \bmzero & \bmzero & \bmzero&\bmzero & \bmzero & \bmzero & \bmzero & \bmone&\bmzero & \bmzero & \bmzero & \bmzero
\end{array}\]
%

The equitable partition with starting hours $[1,1,2,2,3,4,5]$ (on the left in~\Cref{ex:ex1}) is \emph{not} successful because the inspector experiences delays between the end of the third task and the beginning of the second, and between the end of the sixth task and the beginning of the fourth:

\[\begin{array}{cc|c|ccccc|ccccc|ccccc|c}
& & \multicolumn{15}{c}{{\text{Hours}}} \\
& & 5 & 1& 2& 3& 4& 5 & 1& 2& 3& 4& 5 & 1& 2&3& 4 & 5 &1\\ \hline
\multirow{7}{*}{{Tasks}}
&1 & \bmzero&\bmone & \bmzero & \bmzero & \bmzero & \bmzero&\bmzero & \bmzero & \bmzero & \bmzero & \bmzero&\bmzero & \bmzero & \bmzero & \bmzero& \bmzero & \bmzero\\
&2 & \bmzero&\bmzero & \bmzero & \bmzero & \bmzero & \bmzero&\bmone & \bmone & \bmone & \bmzero & \bmzero&\bmzero & \bmzero & \bmzero & \bmzero& \bmzero & \bmzero\\
&3 & \bmzero&\bmzero & \bmone & \bmone & \bmone & \bmzero&\bmzero & \bmzero & \bmzero & \bmzero & \bmzero&\bmzero & \bmzero & \bmzero & \bmzero& \bmzero & \bmzero\\
&4 & \bmzero&\bmzero & \bmzero & \bmzero & \bmzero & \bmzero&\bmzero & \bmzero & \bmzero & \bmzero & \bmzero&\bmzero & \bmone & \bmzero & \bmzero& \bmzero & \bmzero \\
&5 & \bmzero&\bmzero & \bmzero & \bmzero & \bmzero & \bmzero&\bmzero & \bmzero & \bmzero & \bmzero & \bmzero&\bmzero & \bmzero & \bmone & \bmone& \bmone & \bmone\\
&6 & \bmzero&\bmzero & \bmzero & \bmzero & \bmzero & \bmzero&\bmzero & \bmzero & \bmzero & \bmone & \bmone&\bmzero & \bmzero & \bmzero & \bmzero& \bmzero & \bmzero\\
&7 & \bmone&\bmzero & \bmzero & \bmzero & \bmzero & \bmzero&\bmzero & \bmzero & \bmzero & \bmzero & \bmzero&\bmzero & \bmzero & \bmzero & \bmzero& \bmzero & \bmzero
\end{array}\]
\end{example}

It turns out that the successful partition is unique---it is the equitable partition for which each task is started as late as possible (among all equitable partitions).  In an amusing twist to what ``real-life'' experience might suggest, this particular inspector will be satisfied only if the worker procrastinates.

\subsection{The Modular Sweep Map} 
A solution to the scheduling problem above is closely related to inverting a simple and curious map defined on words.

\medskip

As above, we let $m,N \in \mathbb{N}$, but we now write $\Am$ for the set of words of length $N$ on the alphabet $\{0,1,2,\ldots,m-1\}$.  For a word $\w=\w_1 \w_2 \cdots \w_{N} \in \Am$ and for $1 \leq j \leq N$, define the \defn{modular level} of the letter $\w_j$ to be $\ell_j:=\sum_{i=1}^j \w_i \mod m$.

The \defn{modular sweep map} is the function $\sw:\Am \to \Am$ that sorts $\w \in \Am$ according to its modular levels as follows: initialize $\uu=\emptyset$ to be the empty word.  For $k=m-1,\ldots,2,1,0$, read $\w$ from right to left and append to $\uu$ all letters $\w_j$ whose level $\ell_j$ is equal to $k$.  Define $\sw(\w):=\uu$.


\begin{example}
\label{ex:zeta_intro}
Let $m=5$ and $N=7$.  We compute the modular levels of the word $\w = 3113214 \in \Am$ by summing the initial letters of $\w$ modulo $m$ and obtain the image $\uu:=\sw(\w)$ by sorting according to the levels (and then discarding the information about the levels).


\[\begin{array}{rccccccc}
\ell: &3&4&0&3&0&1&0\\
\w: &3&1&1&3&2&1&4
\end{array} \raisebox{0.2em}{$\xmapsto[\sw]{}$} \begin{array}{rccccccc}
\ell: & 4&3&3&1&0&0&0\\
\uu: & 1&3&3&1&4&2&1
\end{array}.\]


Note that $\uu$ records the lengths of the tasks, in order of priority, while $\w$ records the sequence of task lengths in the reverse of the order in which they are watched by the inspector in~\Cref{ex:ex2}.
Indeed, as we shall make precise, finding a sequence of tasks that the inspector can
watch in order with no breaks amounts to inverting the modular sweep map.
\end{example}
Our main result---proven in~\Cref{sec:sweep}---is that $\sw$ is invertible.\footnote{As G.~Warrington pointed out to us at the American Institute of Mathematics in 2012---sorting is not usually an invertible operation!}

\begin{theorem}
\label{thm:main_thm}
The modular sweep map is a bijection $\Am \to \Am$.
\end{theorem}



\subsection{Organization} The remainder of this paper is organized as follows.

We give a brief history in~\Cref{sec:history} by recalling the different contexts in which the modular sweep map has appeared.  

In~\Cref{sec:mod_presweep}, we define the modular presweep map.  This map differs from the modular sweep map in that it preserves the additional information of the modular levels.  Partitioned words in the image of the modular presweep map are called \emph{successful partitions}.  It is easy to invert the modular presweep map, as described in~\Cref{sec:inv_modular_presweep}.

In~\Cref{sec:equ_and_succ}, we introduce the notion of \emph{equitable partitions} and show that a succcessful partition is equitable.  Using an algorithm communicated to us by F.~Aigner, C.~Ceballos, and R.~Sulzgruber (inspired by our related~\Cref{map:rightmost}), we construct the rightmost equitable partition in~\Cref{thm:unique_rightmost}, and we show how to modify any equitable partition to produce a successful partition in~\Cref{thm:unique_succ}.  \Cref{thm:right_eq_succ} concludes that the rightmost equitable partition and successful partition are the same.

We apply the results of~\Cref{sec:modular_presweep,sec:remember} to prove~\Cref{thm:main_thm}---that the modular sweep map is a bijection---in~\Cref{sec:sweep}. 


In analogy to the stable marriage problem of D.~Gale and L.~Shapley, we construct the leftmost equitable partition in~\Cref{lem:rightmost}, and we show that the set of all equitable partitions may be given the structure of a distributive lattice in \Cref{thm:dis_lattice}.  \Cref{sec:eq_lattice} is not needed for the applications of~\Cref{thm:main_thm} given in~\Cref{sec:applications}, and may be skipped.

 In~\Cref{sec:applications}, we use~\Cref{thm:main_thm} to solve two problems from the literature: we show in~\Cref{sec:general_sweep} how to invert the sweep map of~\cite{armstrong2014sweep} on words with letters in $\mathbb{Z}$ (rather than $\mathbb{Z}/m\mathbb{Z}$), and we conclude in~\Cref{sec:dyck} that the zeta map is bijective on Dyck paths and rational Dyck paths.

\section*{Acknowledgements}
The second author is indebted to Vic Reiner and Dennis Stanton for originally suggesting this area to him.

We thank Drew Armstrong, Cesar Ceballos, Adriano Garsia, Nick Loehr, and Greg Warrington for their encouragement and enthusiasm, close reading, detection of typos, and suggestions for improvements to the exposition.  We thank Marko Thiel for pointing out that the argument of~\cite[Proposition 3.2]{armstrong2014sweep} could be extended to arbitrary alphabets in~\Cref{thm:sweep_dyck}.  We warmly thank Florian Aigner, Cesar Ceballos, and Robin Sulzgruber for finding~\Cref{map:leftmost} in the course of their study of an early draft of this paper, and then patiently explaining it to us at FPSAC 2016.  We thank the anonymous referee whose diligence and detailed comments greatly improved the paper.
An extended abstract of this paper was published in the proceedings of FPSAC 2017 \cite{fpsacabs}.

Both authors thank their parents for introducing them to stable marriages.

This research was supported by NSERC and the Canada Research Chairs program.
\section{History}
\label{sec:history}
\subsection{Diagonal Harmonics and the Zeta Map}

In their study of the space $\mathcal{DH}_n$ of diagonal harmonics~\cite{garsia1996remarkable}, A.~Garsia and M.~Haiman defined a rational function $C_n(q,t)$, symmetric in $q$ and $t$, with the property that $C_n(1,1)=\frac{1}{n+1}\binom{2n}{n}$.  They conjectured that $C_n(q,t)$ was actually a \emph{polynomial} in $q$ and $t$ with nonnegative coefficients.  Specializing one of the statistics to $1$, they gave a combinatorial interpretation of this polynomial using the area statistic on \defn{$n$-Dyck paths} (lattice paths from $(0,0)$ to $(n,n)$ that stay above the diagonal $y=x$):

\[C_n(q,1)=C_n(1,q)=\sum_{w \text{ an } n\text{-Dyck path}} q^{\mathrm{area}(w)}.\]

The search was on to find a statistic that manifested nonnegativity---an unknown statistic with the property that

\[C_n(q,t) = \sum_{w \text{ an } n\text{-Dyck path}} q^{\mathrm{area}(w)} t^{\mathrm{unknown}(w)}.\]

In~\cite{haglund2003conjectured}, ``after a prolonged investigation of tables of $C_n(q,t)$,'' J.~Haglund invented the idea of a \defn{bounce path}, which he used to propose exactly such a statistic.    Garsia and Haglund subsequently used these ideas to prove nonnegativity of $C_n(q,t)$ in~\cite{garsia2002proof}.

As the legend goes, Garsia sent a cryptic email to Haiman announcing Haglund's discovery---without providing any specifics as to what the statistic was.  Shortly after, Haiman announced that he, too, had produced the desired statistic.\footnote{Garsia subsequently expressed regret that he didn't send this email to Haiman several years earlier.}  Remarkably, Haglund's statistic and Haiman's statistic were \emph{different}.  In modern language, Haglund's statistic is known as \defn{$\mathrm{bounce}$}, while Haiman's is \defn{$\mathrm{dinv}$}.  Haiman and Haglund quickly developed a bijection from $n$-Dyck paths to themselves---the \defn{zeta map} $\zeta$ (see~\Cref{sec:zeta})~\cite{andrews2002ad,haglund2003conjectured}---such that \[(\mathrm{area}(w),\mathrm{bounce}(w))=(\mathrm{dinv}(\zeta(w)),\mathrm{area}(\zeta(w))).\]

As Dyck paths have been generalized (say, as in~\Cref{sec:dyck}, to lattice paths from $(0,0)$ to $(a,b)$ that stay above the main diagonal), so too have these zeta maps~\cite{loehr2003multivariate,egge2003schroder,gorsky2014compactified,lee2014combinatorics}.  A modern perspective is that there is only one statistic---area---along with a generalized zeta map~\cite{armstrong2014sweep}.  If such a zeta map is bijective on a set of generalized Dyck paths $\mathcal{D}$, one can combinatorially define polynomials
\[ \mathcal{D}(q,t):=\sum_{w\in \mathcal{D}} q^{\mathrm{area(w)}} t^{\mathrm{area(\zeta(w))}},\]
so that (by construction) $\mathcal{D}(q,1)=\mathcal{D}(1,q)$.  Surprisingly, these polynomials also often happen to be symmetric in $q$ and $t$.

Proving invertibility of these generalized zeta maps has been a traditionally difficult problem; combinatorially proving $(q,t)$-symmetry has been intractable.  Most recently, D.~Armstrong, N.~Loehr, and G.~Warrington have found very general versions of the zeta map, which they called \defn{sweep maps}~\cite[Section 3.4]{armstrong2014sweep}.  C.~Reutenauer independently developed a related map in an unpublished letter to Garsia, A.~Hicks, and E.~Leven~\cite{reut2013}.

\subsection{Suter's Cyclic Symmetry}
In~\cite{suter2002young}, R.~Suter defined a striking cyclic symmetry of order $n+1$ on the subposet of Young's lattice consisting of those partitions with largest hook at most $n$.  \Cref{fig:young} illustrates this for $n=4$.

\begin{figure}[htbp]
\vspace{-.7em}
\begin{center}
        \raisebox{-0.5\height}{\includegraphics[height=1.7in]{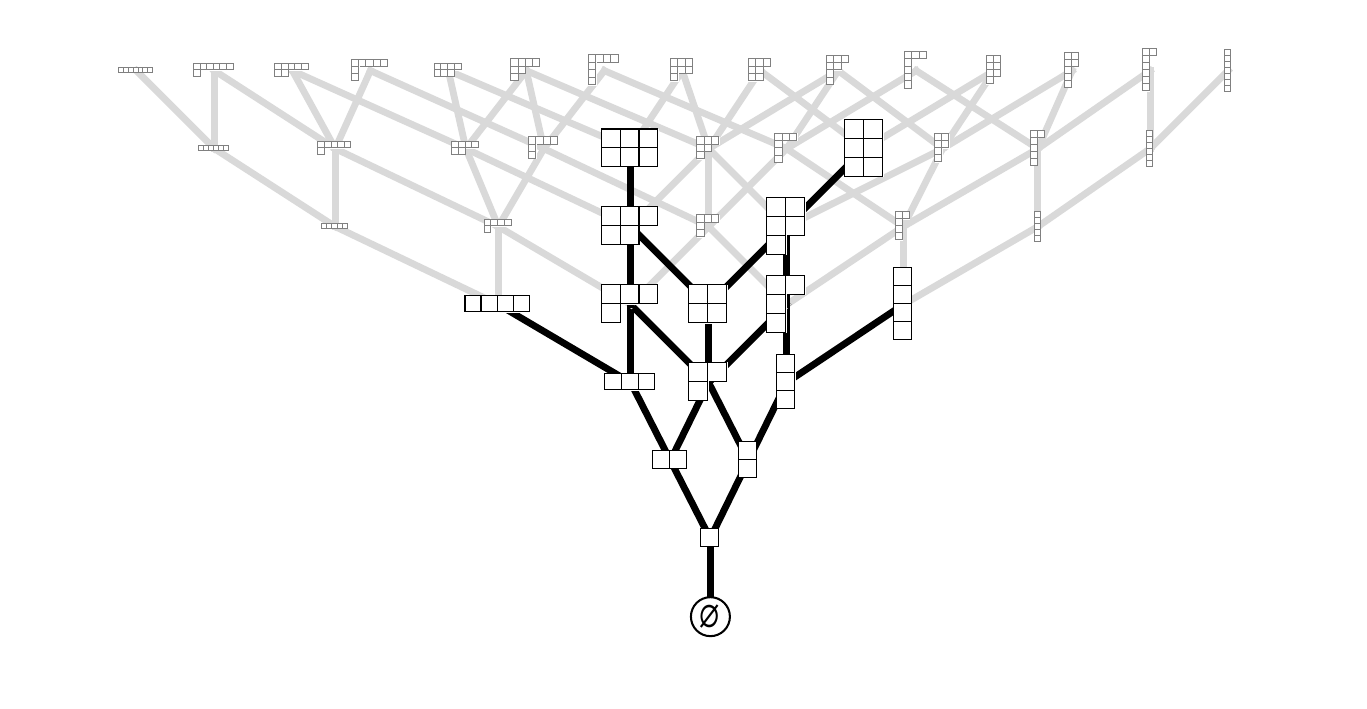}} \hspace{-1em} \raisebox{-0.5\height}{\includegraphics[height=1.7in]{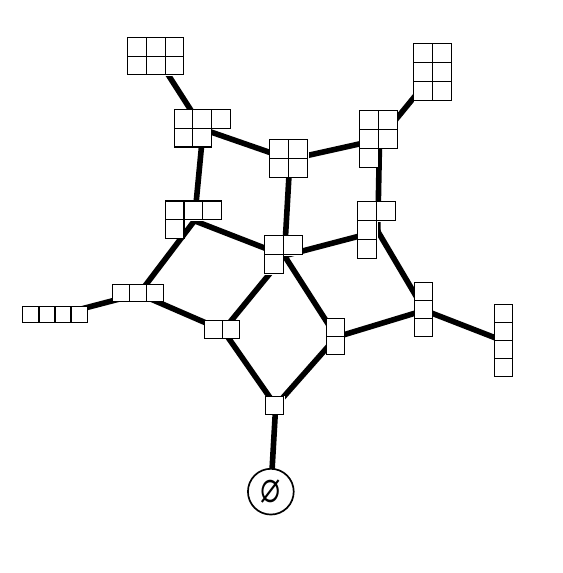}}
\end{center}
\caption{On the left is the subposet of Young's lattice containing those partitions with largest hook at most $4$.  On the right is the underlying graph, drawn to reveal a five-fold symmetry.}
\label{fig:young}
\end{figure}

In a followup paper~\cite{suter2004abelian}, Suter explained this symmetry by relating this subposet to the weak order on the two-fold dilation of the fundamental alcove in affine type $\widetilde{A}_n$.  The second author was introduced to this symmetry by V.~Reiner and D.~Stanton, and gave a combinatorial generalization in~\cite{williamsbijactions} to accommodate a parameter $m$.  The corresponding geometric interpretation was given by M.~Zabrocki and C.~Berg~\cite{berg2011symmetries} as an $m$-fold dilation of the fundamental alcove.  We refer the reader to~\cite[Section 2]{thomas2014cyclic} for a more detailed history.

To characterize the orbit structure of this cyclic symmetry, the second author defined a map to words under rotation via a general technique called a \emph{bijaction}.  It is not hard to show that this map is a bijection when the parameter $m$ is equal to $2$~\cite{williamsbijactions}, and M.~Visontai was able to invert it for $m=3$~\cite{visontaibijection}.  The general case was resistant to all attack---until the authors were introduced at the 2012 Combinatorial Algebra meets Algebraic Combinatorics conference at the Universit\'e du Qu\'ebec \`a Montr\'eal.  Our joint paper~\cite{thomas2014cyclic} resulted.

\medskip

Unexpectedly, it turns out that inverting the general sweep map and inverting the bijaction to describe the symmetry of the fundamental alcove in type $\widetilde{A}_n$ are essentially the same problem.

\section{Presweeping and Its Inverse}
\label{sec:modular_presweep}

We will factor the modular sweep map as the composition of two maps: the modular presweep map and the forgetful map.  In this section, we define the modular presweep map and its inverse.

\subsection{Notation}
Adhering to the notation in~\cite{thomas2014cyclic}, we prefer to think of the modular levels from the introduction as partitioning the word $\uu$ into blocks.  Define a \defn{partitioned word} for $\uu \in \Am$ to be a partition $\uu^*$ of $\uu$ into $m$ words $\uu^*=\bb_{m-1}|\bb_{m-2}|\cdots|\bb_0$---where we use the \defn{block divider} symbol $|$ to separate the blocks---so that $\uu=\bb_{m-1}\cdots \bb_0$ is their concatenation.  We call the word $\bb_k$ the $k$th \defn{block} and, with apologies to the combinatorics of words community, we write $\Am^*$ for the set of all partitioned words of $\Am$.  We call $\uu$ the \defn{underlying word} of the partitioned word $\uu^*$.  We may use either the symbol $\cdot$ or $\emptyset$ to denote an empty block.   If the $i$-th letter $\uu_i$ of $\uu$ belongs to the $k$-th block $\bb_k$ in the partitioned word $\uu^*$, we let $\block(\uu^*,i):=k$.  We fix the notation $|\uu|:=\sum_{i=1}^{N}\uu_i$ and $|\uu|_m = \ell_N = |\uu|\mod m$.

\subsection{The Modular Presweep Map}
\label{sec:mod_presweep}
The \defn{modular presweep map} is the function $\swp:\Am \hookrightarrow \Am^*$ that sorts $\w \in \Am$ into blocks according to its levels.  Precisely, for $k=m-1,\ldots,2,1,0$, first initialize $\uu^*:=\cdot|\cdot|\cdots|\cdot$ to be the empty partitioned word, then read $\w$ from right to left and append to $\bb_k$ all letters $\w_j$ whose level $\ell_j=\left(\sum_{i=1}^j \w_i \mod m\right)$ is equal to $k$.  In other words, $\uu^*_k$ is obtained by extracting all letters of level $k$ from $\uu$ and reversing their relative order.  Pseudo-code for $\swp$ is given in~\Cref{map:forward}.

\begin{algorithm}[htbp]
\KwIn{A word $\w = \w_1 \w_2 \cdots \w_{N} \in \Am$.}
\KwOut{A partitioned word $\uu^*=\bb_{m-1}|\bb_{m-2}|\cdots|\bb_0 \in \Am^*$.}
Let $\uu^*:=\cdot|\cdot|\cdots|\cdot$\;
\For{$k=m-1$ \KwTo $0$}{
\For{$j=N$ \KwTo $1$}{
 \If{$\left(\sum_{i=1}^j \w_i \mod m\right)=k$}{Append $\w_j$ to $\bb_k$\;}
    }
}
Return $\uu^*$\;
\caption{$\swp:\Am \hookrightarrow \Am^*$.}
\label{map:forward}
\end{algorithm}

\begin{example}
As in \Cref{ex:zeta_intro}, let $m=5$, $N=7$, and $\w = 3113214 \in \Am.$  We compute the modular levels of a word $\w \in A$ by summing the initial letters of $\w$ modulo $m$ (below, left).  We compute the modular presweep of $\w$ by sorting by levels, reading $\w$ from right to left.  Placing letters with the same level in a block, we obtain the corresponding partitioned word $\uu^*:=\swp(\w)$ in $\Am^*$ (below, right).

\[\begin{array}{rccccccc}
\ell:&3&4&0&3&0&1&0 \\
\w:&3&1&1&3&2&1&4
\end{array} \raisebox{0.2em}{$\xmapsto[\swp]{}$} \begin{array}{rc|cc|c|c|ccc}
\ell: & 4&3&3&2&1&0&0&0\\
\uu^*:& 1&3&3&\cdot&1&4&2&1
\end{array}\]
\label{ex:zeta}
\end{example}
\begin{example}
To further illustrate our algorithms in the case that $|\w|_m\neq 0$, we will use the running example of the word $\w = 2314341 \in \Am$ for $m=5$ and $N=7$.  In the same way as~\Cref{ex:zeta}, we have

\[\begin{array}{rccccccc}
\ell:&2&0&1&0&3&2&3\\
\w:&2&3&1&4&3&4&1
\end{array} \text{ and } \begin{array}{rc|cc|cc|c|cc}
\ell: & 4&3&3&2&2&1&0&0\\
\uu^*:& \cdot&1&3&4&2&1&4&3
\end{array}.\]





\label{ex:zeta2}
\end{example}

\subsection{The Inverse Modular Presweep Map}
\label{sec:inv_modular_presweep}
The \defn{inverse modular presweep map} is the function $\iswp:\Am^*\to\Am$ such that \[\iswp \circ \swp = {\sf id}_{\Am},\] where ${\sf id}_{\Am}(\w)=\w$ is the identity function on $\Am$.  As explained in~\cite[Section 5.2]{armstrong2014sweep} and in~\cite[Algorithm 2 and Figure 8]{thomas2014cyclic}, if we know how to associate the correct levels to $\uu:=\sw(\w)$, it is easy to reconstruct $\w$.

\begin{algorithm}[htbp]
\KwIn{A partitioned word $\uu^*=\bb_{m-1}|\bb_{m-2}|\cdots|\bb_0 \in \Am^*$.}
\KwOut{A word $\w = \w_1 \w_2 \cdots \w_{N} \in \Am$ or a subword of $\uu^*$.}
Let $\ell_{N}:=\left(\sum_{j=1}^{N} \uu_j \mod m\right)$ and $\w:=\emptyset$\;
\For{$i=1$ \KwTo $N$}{
 \If{$\bb_{\ell_{N-i+1}} \neq \emptyset$}{Remove the first letter of $\bb_{\ell_{N-i+1}}$ and assign it to $\w_{N-i+1}$\;
 Prepend $\w_{N-i+1}$ to $\w$\;
 Let $\ell_{N-i}:=\left(\ell_{N-i+1}-\w_{N-i+1} \mod m\right)$\;}
 \Else{Return $\uu^*$}
    }
Return $\w$\;
\caption{$\iswp:\Am^* \hookrightarrow \Am$.}
\label{map:backwards}
\end{algorithm}

Suppose we have the partitioned word $\uu^*:=\swp(\w)$.  Since the letters of $\w$ were just rearranged to make $\uu^*$, we can determine $\ell_{N}=|\w|_m=|\uu|_m$ from $\uu^*$.  As we swept $\w$ from right to left, the last letter of $\w$ is therefore the first letter in $\bb_{\ell_{N}}$.  Remove this letter from $\uu^*$.
Subtracting this letter from $\ell_{N}$ gives $\ell_{N-1}$, and we obtain the $(N-1)$st letter of $\w$ as the first remaining letter in block $\ell_{N-1}$.  In general, for $i=1,2,\ldots,N$ suppose we have already recovered the last $i-1$ letters of $\w$, removed them from $\uu^*$, and computed $\ell_{N-i+1}$; subtracting the leftmost remaining letter in $\bb_{\ell_{N-i+1}}$ from $\ell_{N-i+1}$ (and removing it from $\bb_{\ell_{N-i+1}}$) gives $\ell_{N-i}$ and recovers $\w_{N-i+1}$.  Pseudo-code for $\iswp$ is given in~\Cref{map:backwards}.

We say that~\Cref{map:backwards} \defn{succeeds} on a partitioned word $\uu^*$ if it returns an element of $\Am$, and we say that it \defn{fails} if it returns an element of $\Am^*$.  Since~\Cref{map:backwards} undoes~\Cref{map:forward} one step at a time, we conclude that $\iswp$ is the left inverse of $\swp$.  Recall that a partitioned word is called successful if it is in the image of the modular presweep map.  Clearly, the word $\uu^*$ is successful precisely if~\Cref{map:backwards} succeeds on it.

\begin{example}
\label{ex:successful_zeta}
To reverse~\Cref{ex:zeta,ex:zeta2},
we compute using~\Cref{map:backwards} as follows (the computation for~\Cref{ex:zeta} is on the left, while the computation for~\Cref{ex:zeta2} is on the right).  Starting with the partitioned words
\vspace{-1em}
\begin{multicols}{2}
\[\begin{array}{rc|cc|c|c|ccc}
\ell: & 4&3&3&2&1&0&0&0\\
\uu^*:& 1&3&3&\cdot&1&4&2&1
\end{array},\]
we find $\ell_N=|\uu|_m=0.$   We iterate:
\[\begin{array}{c|c|c|c}
i&\uu^*&\ell_{N-i}&\w \\ \hline
0&1|33|\cdot|1|\dot{4}21 & 0 & \cdot \\
1&1|33|\cdot|\dot{1}|\textcolor{lightgray}{4}21 & 1 & 4 \\
2&1|33|\cdot|\textcolor{lightgray}{1}|\textcolor{lightgray}{4}\dot{2}1 & 0 & 14 \\
3&1|\dot{3}3|\cdot|\textcolor{lightgray}{1}|\textcolor{lightgray}{42}1 & 3 & 214 \\
4&1|\textcolor{lightgray}{3}3|\cdot|\textcolor{lightgray}{1}|\textcolor{lightgray}{42}\dot{1} & 0 & 3214 \\
5&\dot{1}|\textcolor{lightgray}{3}3|\cdot|\textcolor{lightgray}{1}|\textcolor{lightgray}{421} & 4 & 13214 \\
6&\textcolor{lightgray}{1}|\textcolor{lightgray}{3}\dot{3}|\cdot|\textcolor{lightgray}{1}|\textcolor{lightgray}{421} & 3 & 113214 \\
7&\textcolor{lightgray}{1}|\textcolor{lightgray}{33}|\cdot|\textcolor{lightgray}{1}|\textcolor{lightgray}{421} & 0 & 3113214
\end{array}\]
Comparing with~\Cref{ex:zeta}, we see that we have recovered $\w$.

\columnbreak
\[\begin{array}{rc|cc|cc|c|cc}
\ell: & 4&3&3&2&2&1&0&0\\
\uu^*:& \cdot&1&3&4&2&1&4&3
\end{array},\]
we find $\ell_N=|\uu|_m=3.$  We iterate:
\[\begin{array}{c|c|c|c}
i&\uu^*&\ell_{N-i}&\w \\ \hline
0&\cdot|\dot{1}3|42|1|43 & 3 & \cdot \\
1&\cdot|\textcolor{lightgray}{1}3|\dot{4}2|1|43 & 2 & 1 \\
2&\cdot|\textcolor{lightgray}{1}\dot{3}|\textcolor{lightgray}{4}2|1|43 & 3 & 41 \\
3&\cdot|\textcolor{lightgray}{13}|\textcolor{lightgray}{4}2|1|\dot{4}3 & 0 & 341 \\
4&\cdot|\textcolor{lightgray}{13}|\textcolor{lightgray}{4}2|\dot{1}|\textcolor{lightgray}{4}3 & 1 & 4341 \\
5&\cdot|\textcolor{lightgray}{13}|\textcolor{lightgray}{4}2|\textcolor{lightgray}{1}|\textcolor{lightgray}{4}\dot{3} & 0 & 14341 \\
6&\cdot|\textcolor{lightgray}{13}|\textcolor{lightgray}{4}\dot{2}|\textcolor{lightgray}{1}|\textcolor{lightgray}{43} & 2 & 314341 \\
7&\cdot|\textcolor{lightgray}{13}|\textcolor{lightgray}{42}|\textcolor{lightgray}{1}|\textcolor{lightgray}{43} & 0 & 2314341
\end{array}\]
Comparing with~\Cref{ex:zeta2}, we see that we have recovered $\w$.
\end{multicols}
\end{example}

\begin{example}
\label{ex:unsuccessful_zeta}
Using different partitioned words for the underlying words in~\Cref{ex:zeta,ex:zeta2}, we give two examples of when~\Cref{map:backwards} fails.
\vspace{-1em}
\begin{multicols}{2}

\[\begin{array}{c|c|c|c}
i&\uu^*&\ell_{N-i}&\w \\ \hline
0&13|31|4|2|\dot{1} & 0 & \cdot \\
1&\dot{1}3|31|4|2|\textcolor{lightgray}{1} & 4 & 1 \\
2&\textcolor{lightgray}{1}3|\dot{3}1|4|2|\textcolor{lightgray}{1} & 3 & 11 \\
3&\textcolor{lightgray}{1}3|\textcolor{lightgray}{3}1|4|2|\textcolor{lightgray}{1} & 0 & 311
\end{array}\]


\[\begin{array}{c|c|c|c}
i&\uu^*&\ell_{N-i}&\w \\ \hline
0&1|\dot{3}4|21|4|3 & 3 & \cdot \\
1&1|\textcolor{lightgray}{3}4|21|4|\dot{3} & 0 & 3 \\
2&1|\textcolor{lightgray}{3}4|\dot{2}1|4|\textcolor{lightgray}{3} & 2 & 33 \\
3&1|\textcolor{lightgray}{3}4|\textcolor{lightgray}{2}1|4|\textcolor{lightgray}{3} & 0 & 233
\end{array}\]
\end{multicols}

\Cref{map:backwards} returns the remaining letters of $\uu^*$.  The blocks of this partitioned word are suffixes of the blocks of the original $\uu^*$.
\end{example}

\subsection{Forgetting}
\label{sec:forget}
We now obtain the modular sweep map from the modular pre\-sweep map by forgetting the information of the blocks.  The \defn{forgetful map} is the function \begin{align*}\f:\Am^* &\to \Am \\ \f\left(\bb_{m-1}|\bb_{m-2}|\cdots|\bb_0\right)&=\bb_{m-1}\bb_{m-2}\cdots\bb_0\end{align*} obtained by concatenating all the blocks of $\uu^* \in \Am^*$.  Thus, the modular sweep map of~\Cref{sec:intro} may be written as the composition \[\sw = \left(\f \circ \swp\right): \Am \to \Am.\]

\begin{example}
Continuing with~\Cref{ex:zeta}, we forget the partitioning to obtain the modular sweep $\uu:=\sw(\w)$ of $\w$ to be
\begin{align*}(\f \circ \swp)(3113124)&=\f\left(1|33|\cdot|1|421\right) = 1331421.\end{align*}

Similarly, for~\Cref{ex:zeta2}, we obtain
\begin{align*}(\f \circ \swp)(2314341)&=\f\left(\cdot|13|42|1|43\right) = 1342143.\end{align*}
\end{example}

Thus, the problem of inverting the modular sweep map has been reduced to showing that there exists a unique successful partition $\uu^* \in \Am^*$ for each word $\uu \in \Am$.  We do this in the next section.

\section{Equitable Partitions and the Successful Partition}
\label{sec:remember}
We already solved the problem of constructing the successful partition in~\cite{thomas2014cyclic}, where we studied a composition \[{f} \circ p,\] where $f$ is the map $\f$ and $p$ is a map \emph{very slightly} different from $\swp$ (see~\Cref{sec:diffs}).  In particular, our notions here of a successful partition and the forgetful map coincide with those in~\cite{thomas2014cyclic}.

Our algorithm for inverting $\f$ uses the notion of  \emph{equitable partitions}, and is considerably streamlined using an idea of F.~Aigner, C.~Ceballos, and R.~Sulzgruber and the helpful suggestions of an anonymous referee.



\subsection{The Poset of Equitable Partitions}\label{sec:equ_and_succ}
 We expand a partitioned word $\uu^*$ into an $N \times m$ \defn{balancing array} \[M^{\uu^*}=(M^{\uu^*}_{i,j})_{\substack{1\leq i \leq N \\ m-1\geq j \geq 0}}\] defined by
\[M^{\uu^*}_{i,j}:=\begin{cases} \bmone & \text{if } j \in \{ \block(\uu^*,i),\block(\uu^*,i)-1,\ldots,\block(\uu^*,i)-\uu_i+1 \} \mod m, \\ \bmzero & \text{otherwise,} \end{cases}\]

Write $|\uu| = \sum_{i=1}^N u_i = q m +r$ with $0 \leq r < m$.  We say that column $j$ (for $m-1 \geq j \geq 0$) of $M^{\uu^*}$ is \defn{equitably filled} if:
\begin{itemize}
\item $r \geq j \geq 1$ and column $j$ has $q+1$ copies of the symbol $\bmone$, or
\item $j=0$ or $j>r$, and column $j$ contains $q$ copies of $\bmone$.
\end{itemize}

In particular, if $r=0$, then every equitably filled column has $q$ copies of $\bmone$.  We say that $\uu^*$ is an \defn{equitable partition} if each of the columns of $M^{\uu^*}$ is equitably filled.  If a column has (strictly) fewer copies of the symbol $\bmone$ than it would to be equitably filled, we say it is \defn{less than equitably filled}; similarly, when a column has (strictly) more copies of $\bmone$ we say that it is \defn{more than equitably filled}.    

The notion of equitable partitions generalizes the ``real world'' example given in~\Cref{sec:real_world}.  The difference is that we no longer make the simplifying assumption that $|\uu|$ is divisible by
$m$.  Examples of balancing arrays are given in~\Cref{ex:successful_def,ex:successful_def2} for the partitioned words shown in~\Cref{ex:componentwise}.  Note that the labeling of columns has changed from~\Cref{ex:ex1}.

  The motivation for the definition of equitable partitions is the following lemma.

\begin{lemma}
\label{lem:succ_equ}
Any successful partition $\uu^*$ is an equitable partition.
\end{lemma}

\begin{proof}
We can construct all successful partitions as follows~\cite[Definition 7.5]{thomas2014cyclic}.  Define an infinite complete $m$-ary tree $\mathcal{T}_m^*$ by
\begin{enumerate}
	\item The zeroth rank consists of the empty successful partition $\uu^*$, given by $\bb_k=\emptyset$ for $m-1 \geq k \geq 0$.
	\item The children of a successful partition $\uu^*=\bb_{m-1}|\ldots|\bb_0$ are the $m$ successful partitions obtained by prepending $i \pmod m$ to $\bb_{i+|\uu|_m \mod m}$ for $i=0,1,2,\ldots,m-1$.
\end{enumerate}

These children are also successful partitions because the addition of the new letter $i$ to $\bb_{i+|\uu|_m}$ changes the sum $|\uu^*|_m$ so that the first letter read in \Cref{map:backwards} is this new letter; \Cref{map:backwards} then reads the remaining letters of the successful parent $\uu^*$ in the same order as before the letter $i$ was added.  In particular, if $\w$ is the word whose image under the modular presweep map is a node $\uu^*$, then the children of $\uu^*$ are characterized as those words obtained by adding one additional letter to the end of $\w$ and applying $\swp$ to the resulting words.  An example of a node in $\mathcal{T}^*_5$ with its five children is given on the right of~\Cref{ex:tree}.  The corresponding words under \Cref{map:backwards} are illustrated on the left of~\Cref{ex:tree}.

Then it is easy to see that all partitioned words in $\mathcal{T}^*_m$ are equitable, and that all successful partitioned words appear in $\mathcal{T}^*_m$~\cite[Lemma 7.2]{thomas2014cyclic}. (It is not yet clear that the images of the words in $\mathcal{T}^*_m$ under the forgetful map are actually distinct.)
\end{proof}

\begin{example}
\label{ex:tree}
Below on the right is the node $\uu^*=\cdot|13|42|1|43$ in $\mathcal{T}^*_5$ with its five children.  For each child, the new letter added to $\uu^*$ is indicated in bold.  On the left, arranged in a similar tree structure, are the words whose images under the modular presweep map give $\uu^*$ and its children.  Observe on the left-hand side that the new letters are simply appended to the word.  On the right-hand side, new letters are prepended to blocks, and then seem to be lost in the middle of the word when the partition is forgotten.

\[
\raisebox{-0.5\height}{\scalebox{0.7}{
\begin{tikzpicture}
\draw (1,1) -- (0,0);
\draw (0,0) --(-3,-1);
\draw (0,0) --(-1.5,-1.4);
\draw (0,0) --(0,-1.8);
\draw (0,0) --(1.5,-2.2);
\draw (0,0) --(3,-2.6);
\node [fill=white!20,rounded corners=2pt,inner sep=1pt] at (0,0) {$2314341$};
\node [fill=white!20,rounded corners=2pt,inner sep=1pt] at (-3,-1) {$2314341\mathbf{0}$};
\node [fill=white!20,rounded corners=2pt,inner sep=1pt] at (-1.5,-1.4) {$2314341\mathbf{1}$};
\node [fill=white!20,rounded corners=2pt,inner sep=1pt] at (0,-1.8) {$2314341\mathbf{2}$};
\node [fill=white!20,rounded corners=2pt,inner sep=1pt] at (1.5,-2.2) {$2314341\mathbf{3}$};
\node [fill=white!20,rounded corners=2pt,inner sep=1pt] at (3,-2.6) {$2314341\mathbf{4}$};
\end{tikzpicture}}} \xmapsto[\swp]{}
\raisebox{-0.5\height}{\scalebox{0.7}{
\begin{tikzpicture}
\draw (1,1) -- (0,0);
\draw (0,0) --(-3,-1);
\draw (0,0) --(-1.5,-1.4);
\draw (0,0) --(0,-1.8);
\draw (0,0) --(1.5,-2.2);
\draw (0,0) --(3,-2.6);
\node [fill=white!20,rounded corners=2pt,inner sep=1pt] at (0,0) {$\cdot|13|42|1|43$};
\node [fill=white!20,rounded corners=2pt,inner sep=1pt] at (-3,-1) {$\cdot|\mathbf{0}13|42|1|43$};
\node [fill=white!20,rounded corners=2pt,inner sep=1pt] at (-1.5,-1.4) {$\mathbf{1}|13|42|1|43$};
\node [fill=white!20,rounded corners=2pt,inner sep=1pt] at (0,-1.8) {$\cdot|13|42|1|\mathbf{2}43$};
\node [fill=white!20,rounded corners=2pt,inner sep=1pt] at (1.5,-2.2) {$\cdot|13|42|\mathbf{3}1|43$};
\node [fill=white!20,rounded corners=2pt,inner sep=1pt] at (3,-2.6) {$\cdot|13|\mathbf{4}42|1|43$};
\end{tikzpicture}}}\]
\end{example}


We now define the \defn{componentwise order} on the set of all equitable partitions of a fixed word $\uu$.

\begin{definition}
Fix $\uu \in \Am$.  Let $\uu^*$ and $\vv^*$ be two equitable partitioned words of $\uu$ and define a partial order $\uu^*\leq \vv^{*}$ if and only if $\block(\uu^*,i)\geq \block(\vv^{*},i)$ for all $1 \leq i \leq N$.  We denote the resulting poset on all equitable partitioned words of $\uu$ by $\equ(\uu)$.
\end{definition}

\begin{example}
\label{ex:componentwise}
Fix $m=5$.  The posets $\equ(1331421)$ and $\equ(1342143)$ appear below, and should be compared with the balancing arrays in~\Cref{ex:successful_def,ex:successful_def2}.  The maximum element is what we will later call in~\Cref{def:rightmost} the rightmost equitable partition; the minimum element is what we will later call in~\Cref{def:leftmost} the leftmost equitable partition.  For each equitable partition, the letters read by~\Cref{map:backwards} are decorated with a dot.  The underlinings and wavy underlinings should be ignored for the moment---they indicate letters of minimal left balanced block suffixes, as defined in~\Cref{def:block_suffix}.


\begin{center}
\raisebox{-0.5\height}{\begin{tikzpicture}
\draw (0,0) --(-1.5,1) --(0,2) --(0,3);
\draw (0,0) --(1.5,1) --(0,2);
\node [fill=white!20,rounded corners=2pt,inner sep=1pt] at (0,0) {$\rightmost(\uu)=\dot{1}\uwave{3}|\dot{3}\underline{1}|\underline{4}|\uwave{2}|\dot{1}$};
\node [fill=white!20,rounded corners=2pt,inner sep=1pt] at (0,2) {$\dot{1}|\dot{3}\dot{3}|\underline{1}|\underline{4}|\dot{2}\dot{1}$};
\node [fill=white!20,rounded corners=2pt,inner sep=1pt] at (0,3) {$\invf(\uu)=\dot{1}|\dot{3}\dot{3}|\cdot|\dot{1}|\dot{4}\dot{2}\dot{1}$};
\node [fill=white!20,rounded corners=2pt,inner sep=1pt] at (-1.5,1) {$\dot{1}|\dot{3}\dot{3}\underline{1}|\underline{4}|\cdot|\dot{2}\dot{1}$};
\node [fill=white!20,rounded corners=2pt,inner sep=1pt] at (1.5,1) {$\dot{1}\uwave{3}|\dot{3}|1|4\uwave{2}|\dot{1}$};
\end{tikzpicture}}\hfill
\raisebox{-0.5\height}{\begin{tikzpicture}
\draw (0,-2) -- (0,-1) -- (0,0) --(-1.5,1) --(0,2);
\draw (0,0) --(1.5,1) --(0,2);
\node [fill=white!20,rounded corners=2pt,inner sep=1pt] at (0,-2) {$\rightmost(\uu)=1|\dot{3}42\underline{1}|\underline{4}|3|\cdot$};
\node [fill=white!20,rounded corners=2pt,inner sep=1pt] at (0,0) {$\underline{1}|\dot{3}\underline{4}|\dot{2}\uwave{1}|\uwave{4}|\dot{3}$};
\node [fill=white!20,rounded corners=2pt,inner sep=1pt] at (0,-1) {$1|\dot{3}4\underline{2}|1|4\underline{3}|\cdot$};
\node [fill=white!20,rounded corners=2pt,inner sep=1pt] at (0,2) {$\invf(\uu)=\cdot|\dot{1}\dot{3}|\dot{4}\dot{2}|\dot{1}|\dot{4}\dot{3}$};
\node [fill=white!20,rounded corners=2pt,inner sep=1pt] at (-1.5,1) {$\underline{1}|\dot{3}\underline{4}|\dot{2}|\dot{1}|\dot{4}\dot{3}$};
\node [fill=white!20,rounded corners=2pt,inner sep=1pt] at (1.5,1) {$\cdot|\dot{1}\dot{3}|\dot{4}\dot{2}\uwave{1}|\uwave{4}|\dot{3}$};
\end{tikzpicture}}
\end{center}
\end{example}

\begin{example}
Fix $m=5$.  The balancing arrays for the poset $\equ(1331421)$ are arranged to correspond with the diagram on the left of~\Cref{ex:componentwise}.  Since $|\uu|=3 \cdot 5 + 0$, every equitable filling has three copies of $\bmone$ in each column $j$.


\begin{center}
\scalebox{0.8}{
\begin{tikzpicture}[scale=3.5]
\draw (0,0) --(-1.5,1) --(0,2) --(0,4);
\draw (0,0) --(1.5,1) --(0,2);
\node [fill=white!20,rounded corners=2pt,inner sep=1pt] at (0,0) {$\begin{array}{cc|ccccc}
& & \multicolumn{5}{c}{{j}} \\
& & 4 & 3& 2& 1& 0\\ \hline
\multirow{7}{*}{{$i$}}
&1& \bmoner & \bmzero & \bmzero & \bmzero & \bmzero \\
&2&\bmoner & \bmone & \bmone & \bmzero & \bmzero \\
&3&\bmzero & \bmoner & \bmone & \bmone & \bmzero\\
&4&\bmzero & \bmoner & \bmzero & \bmzero & \bmzero \\
&5&\bmone & \bmzero & \bmoner & \bmone & \bmone \\
&6&\bmzero & \bmzero & \bmzero & \bmoner & \bmone \\
&7&\bmzero & \bmzero & \bmzero & \bmzero & \bmoner
\end{array}$};
\node [fill=white!20,rounded corners=2pt,inner sep=1pt] at (0,2) {$\begin{array}{cc|ccccc}
& & \multicolumn{5}{c}{{j}} \\
& & 4 & 3& 2& 1& 0\\ \hline
\multirow{7}{*}{{$i$}}
&1& \bmoner & \bmzero & \bmzero & \bmzero & \bmzero \\
&2&\bmzero & \bmoner & \bmone & \bmone & \bmzero \\
&3&\bmzero & \bmoner & \bmone & \bmone & \bmzero\\
&4&\bmzero & \bmzero & \bmoner & \bmzero & \bmzero \\
&5&\bmone & \bmone & \bmzero & \bmoner & \bmone \\
&6&\bmone & \bmzero & \bmzero & \bmzero & \bmoner \\
&7&\bmzero & \bmzero & \bmzero & \bmzero & \bmoner
\end{array}$};
\node [fill=white!20,rounded corners=2pt,inner sep=1pt] at (0,3.5) {$\begin{array}{cc|ccccc}
& & \multicolumn{5}{c}{{j}} \\
& & 4 & 3& 2& 1& 0\\ \hline
\multirow{7}{*}{{$i$}}
&1& \bmoner & \bmzero & \bmzero & \bmzero & \bmzero \\
&2&\bmzero & \bmoner & \bmone & \bmone & \bmzero \\
&3&\bmzero & \bmoner & \bmone & \bmone & \bmzero\\
&4&\bmzero & \bmzero & \bmzero & \bmoner & \bmzero \\
&5&\bmone & \bmone & \bmone & \bmzero & \bmoner \\
&6&\bmone & \bmzero & \bmzero & \bmzero & \bmoner \\
&7&\bmzero & \bmzero & \bmzero & \bmzero & \bmoner
\end{array}$};
\node [fill=white!20,rounded corners=2pt,inner sep=1pt] at (-1.5,1) {$\begin{array}{cc|ccccc}
& & \multicolumn{5}{c}{{j}} \\
& & 4 & 3& 2& 1& 0\\ \hline
\multirow{7}{*}{{$i$}}
&1& \bmoner & \bmzero & \bmzero & \bmzero & \bmzero \\
&2&\bmzero & \bmoner & \bmone & \bmone & \bmzero \\
&3&\bmzero & \bmoner & \bmone & \bmone & \bmzero\\
&4&\bmzero & \bmoner & \bmzero & \bmzero & \bmzero \\
&5&\bmone & \bmzero & \bmoner & \bmone & \bmone \\
&6&\bmone & \bmzero & \bmzero & \bmzero & \bmoner \\
&7&\bmzero & \bmzero & \bmzero & \bmzero & \bmoner
\end{array}$};
\node [fill=white!20,rounded corners=2pt,inner sep=1pt] at (1.5,1) {$\begin{array}{cc|ccccc}
& & \multicolumn{5}{c}{{j}} \\
& & 4 & 3& 2& 1& 0\\ \hline
\multirow{7}{*}{{$i$}}
&1& \bmoner & \bmzero & \bmzero & \bmzero & \bmzero \\
&2&\bmoner & \bmone & \bmone & \bmzero & \bmzero \\
&3&\bmzero & \bmoner & \bmone & \bmone & \bmzero\\
&4&\bmzero & \bmzero & \bmoner & \bmzero & \bmzero \\
&5&\bmone & \bmone & \bmzero & \bmoner & \bmone \\
&6&\bmzero & \bmzero & \bmzero & \bmoner & \bmone \\
&7&\bmzero & \bmzero & \bmzero & \bmzero & \bmoner
\end{array}$};
\end{tikzpicture}}
\end{center}
\label{ex:successful_def}
\end{example}

\begin{example}
\label{ex:successful_def2}
Fix $m=5$.  The balancing arrays for the poset $\equ(1342143)$ are arranged to correspond with the diagram on the right of~\Cref{ex:componentwise}.   Since $|\uu|=3\cdot 5 +3$, any equitable filling has three copies of $\bmone$ in columns $0$ and $4$, and four copies of $\bmone$ in columns $1$, $2$, and $3$.

\begin{center}
\scalebox{0.8}{\begin{tikzpicture}[scale=3.5]
\draw (0,-3) -- (0,-1.5) -- (0,0) --(-1.5,1) --(0,2);
\draw (0,0) --(1.5,1) --(0,2);
\node [fill=white!20,rounded corners=2pt,inner sep=1pt] at (0,-3) {$\begin{array}{cc|ccccc}
& & \multicolumn{5}{c}{{j}} \\
& & 4 & 3& 2& 1& 0\\ \hline
\multirow{7}{*}{{$i$}}
&1& \bmoner & \bmzero & \bmzero & \bmzero & \bmzero \\
&2&\bmzero & \bmoner & \bmone & \bmone & \bmzero \\
&3&\bmzero & \bmoner & \bmone & \bmone & \bmone\\
&4&\bmzero & \bmoner & \bmone & \bmzero & \bmzero \\
&5&\bmzero & \bmoner & \bmzero & \bmzero & \bmzero \\
&6&\bmone & \bmzero & \bmoner & \bmone & \bmone \\
&7&\bmone & \bmzero & \bmzero & \bmoner & \bmone
\end{array}$};
\node [fill=white!20,rounded corners=2pt,inner sep=1pt] at (0,0) {$\begin{array}{cc|ccccc}
& & \multicolumn{5}{c}{{j}} \\
& & 4 & 3& 2& 1& 0\\ \hline
\multirow{7}{*}{{$i$}}
&1& \bmoner & \bmzero & \bmzero & \bmzero & \bmzero \\
&2&\bmzero & \bmoner & \bmone & \bmone & \bmzero \\
&3&\bmzero & \bmoner & \bmone & \bmone & \bmone\\
&4&\bmzero & \bmzero & \bmoner & \bmone & \bmzero \\
&5&\bmzero & \bmzero & \bmoner & \bmzero & \bmzero \\
&6&\bmone & \bmone & \bmzero & \bmoner & \bmone \\
&7&\bmone & \bmone & \bmzero & \bmzero & \bmoner
\end{array}$};
\node [fill=white!20,rounded corners=2pt,inner sep=1pt] at (0,-1.5) {$\begin{array}{cc|ccccc}
& & \multicolumn{5}{c}{{j}} \\
& & 4 & 3& 2& 1& 0\\ \hline
\multirow{7}{*}{{$i$}}
&1& \bmoner & \bmzero & \bmzero & \bmzero & \bmzero \\
&2&\bmzero & \bmoner & \bmone & \bmone & \bmzero \\
&3&\bmzero & \bmoner & \bmone & \bmone & \bmone\\
&4&\bmzero & \bmoner & \bmone & \bmzero & \bmzero \\
&5&\bmzero & \bmzero & \bmoner & \bmzero & \bmzero \\
&6&\bmone & \bmone & \bmzero & \bmoner & \bmone \\
&7&\bmone & \bmzero & \bmzero & \bmoner & \bmone
\end{array}$};
\node [fill=white!20,rounded corners=2pt,inner sep=1pt] at (0,2) {$\begin{array}{cc|ccccc}
& & \multicolumn{5}{c}{{j}} \\
& & 4 & 3& 2& 1& 0\\ \hline
\multirow{7}{*}{{$i$}}
&1& \bmzero & \bmoner & \bmzero & \bmzero & \bmzero \\
&2&\bmzero & \bmoner & \bmone & \bmone & \bmzero \\
&3&\bmone & \bmzero & \bmoner & \bmone & \bmone\\
&4&\bmzero & \bmzero & \bmoner & \bmone & \bmzero \\
&5&\bmzero & \bmzero & \bmzero & \bmoner & \bmzero \\
&6&\bmone & \bmone & \bmone & \bmzero & \bmoner \\
&7&\bmone & \bmone & \bmzero & \bmzero & \bmoner
\end{array}$};
\node [fill=white!20,rounded corners=2pt,inner sep=1pt] at (-1.5,1) {$\begin{array}{cc|ccccc}
& & \multicolumn{5}{c}{{j}} \\
& & 4 & 3& 2& 1& 0\\ \hline
\multirow{7}{*}{{$i$}}
&1& \bmoner & \bmzero & \bmzero & \bmzero & \bmzero \\
&2&\bmzero & \bmoner & \bmone & \bmone & \bmzero \\
&3&\bmzero & \bmoner & \bmone & \bmone & \bmone\\
&4&\bmzero & \bmzero & \bmoner & \bmone & \bmzero \\
&5&\bmzero & \bmzero & \bmzero & \bmoner & \bmzero \\
&6&\bmone & \bmone & \bmone & \bmzero & \bmoner \\
&7&\bmone & \bmone & \bmzero & \bmzero & \bmoner
\end{array}$};
\node [fill=white!20,rounded corners=2pt,inner sep=1pt] at (1.5,1) {$\begin{array}{cc|ccccc}
& & \multicolumn{5}{c}{{j}} \\
& & 4 & 3& 2& 1& 0\\ \hline
\multirow{7}{*}{{$i$}}
&1& \bmzero & \bmoner & \bmzero & \bmzero & \bmzero \\
&2&\bmzero & \bmoner & \bmone & \bmone & \bmzero \\
&3&\bmone & \bmzero & \bmoner & \bmone & \bmone\\
&4&\bmzero & \bmzero & \bmoner & \bmone & \bmzero \\
&5&\bmzero & \bmzero & \bmoner & \bmzero & \bmzero \\
&6&\bmone & \bmone & \bmzero & \bmoner & \bmone \\
&7&\bmone & \bmone & \bmzero & \bmzero & \bmoner
\end{array}$};
\end{tikzpicture}}
\end{center}
\end{example}


\subsection{The Rightmost Equitable Partition}\label{sec:rightmost}

\begin{definition}
\label{def:rightmost}
A \defn{rightmost equitable partition} is a maximal element of $\equ(\uu)$.
\end{definition}
That is, a rightmost equitable partition is an equitable partition $\uu^*$ of $\uu$ such that any other equitable partition $\vv^{*}$ of $\uu$ has $\block(\vv^{*},i)\geq \block(\uu^*,i)$ for all $i$.

As we will show in~\Cref{thm:unique_rightmost},~\Cref{map:leftmost} explicitly constructs the unique rightmost equitable partition.\footnote{\Cref{map:leftmost} was communicated to us by F.~Aigner, C.~Ceballos, and R.~Sulzgruber; although we had suspected the existence of~\Cref{map:leftmost} (roughly dual to our~\Cref{map:rightmost}) from our analogy with stable marriages (see~\Cref{rem:stable_marriages}), we had been unable to produce it.}  \Cref{thm:right_eq_succ} then proves that the unique rightmost equitable partition of $\uu$ is also the unique successful partition of $\uu$.

\begin{algorithm}[htbp]
\KwIn{A word $\uu \in \Am$.}
\KwOut{The rightmost equitable partition $\uu^* \in \Am^*$.}
Set $\uu^*=\cdot|\cdot|\cdots|\uu$\;
\While{$\uu^*$ is not an equitable partition}{
		Let $j$ be the rightmost column of $M^{\uu^*}$ that is less than equitably filled\;
		Delete the leftmost letter of $\bb_{j-1}$ and append it to $\bb_{j}$\;
}
Return $\uu^*$\;
\caption{$\invf: \Am \to \Am^*$.}
\label{map:leftmost}
\end{algorithm}

\begin{example}
We illustrate~\Cref{map:leftmost} applied to the word $\uu=1331421$.  At each step, the rightmost column with less than its equitable filling is highlighted.


\begin{alignat*}{4}\scalebox{0.6}{$\begin{array}{ccc>{\columncolor{red!20}}cc}
\bmzero & \bmzero & \bmzero & \makebox[\widthof{$\bmone$}][c]{$\bmzero$} & \bmoner \\
\bmone & \bmone & \bmzero & \bmzero & \bmoner \\
\bmone & \bmone & \bmzero & \bmzero & \bmoner\\
\bmzero & \bmzero & \bmzero & \bmzero & \bmoner \\
\bmone & \bmone & \bmone & \bmzero & \bmoner \\
\bmone & \bmzero & \bmzero & \bmzero & \bmoner \\
\bmzero & \bmzero & \bmzero & \bmzero & \bmoner
\end{array}$} &&\to \scalebox{0.6}{$\begin{array}{ccc>{\columncolor{red!20}}cc}
\bmzero & \bmzero & \bmzero & \bmoner & \bmzero \\
\bmone & \bmone & \bmzero & \bmzero & \bmoner \\
\bmone & \bmone & \bmzero & \bmzero & \bmoner\\
\bmzero & \bmzero & \bmzero & \bmzero & \bmoner \\
\bmone & \bmone & \bmone & \bmzero & \bmoner \\
\bmone & \bmzero & \bmzero & \bmzero & \bmoner \\
\bmzero & \bmzero & \bmzero & \bmzero & \bmoner
\end{array}$} &&\to \scalebox{0.6}{$\begin{array}{ccc>{\columncolor{red!20}}cc}
\bmzero & \bmzero & \bmzero & \bmoner & \bmzero \\
\bmone & \bmzero & \bmzero & \bmoner & \bmone \\
\bmone & \bmone & \bmzero & \bmzero & \bmoner\\
\bmzero & \bmzero & \bmzero & \bmzero & \bmoner \\
\bmone & \bmone & \bmone & \bmzero & \bmoner \\
\bmone & \bmzero & \bmzero & \bmzero & \bmoner \\
\bmzero & \bmzero & \bmzero & \bmzero & \bmoner
\end{array}$} &&\to \scalebox{0.6}{$\begin{array}{cc>{\columncolor{red!20}}ccc}
\bmzero & \bmzero & \bmzero & \bmoner & \bmzero \\
\bmone & \bmzero & \bmzero & \bmoner & \bmone \\
\bmone & \bmzero & \bmzero & \bmoner & \bmone\\
\bmzero & \bmzero & \bmzero & \bmzero & \bmoner \\
\bmone & \bmone & \bmone & \bmzero & \bmoner \\
\bmone & \bmzero & \bmzero & \bmzero & \bmoner \\
\bmzero & \bmzero & \bmzero & \bmzero & \bmoner
\end{array}$} \to \\
\scalebox{0.6}{$\begin{array}{ccc>{\columncolor{red!20}}cc}
\bmzero & \bmzero & \bmoner & \bmzero & \bmzero \\
\bmone & \bmzero & \bmzero & \bmoner & \bmone \\
\bmone & \bmzero & \bmzero & \bmoner & \bmone\\
\bmzero & \bmzero & \bmzero & \bmzero & \bmoner \\
\bmone & \bmone & \bmone & \bmzero & \bmoner \\
\bmone & \bmzero & \bmzero & \bmzero & \bmoner \\
\bmzero & \bmzero & \bmzero & \bmzero & \bmoner
\end{array}$} &&\to \scalebox{0.6}{$\begin{array}{cc>{\columncolor{red!20}}ccc}
\bmzero & \bmzero & \bmoner & \bmzero & \bmzero \\
\bmone & \bmzero & \bmzero & \bmoner & \bmone \\
\bmone & \bmzero & \bmzero & \bmoner & \bmone\\
\bmzero & \bmzero & \bmzero & \bmoner & \bmzero \\
\bmone & \bmone & \bmone & \bmzero & \bmoner \\
\bmone & \bmzero & \bmzero & \bmzero & \bmoner \\
\bmzero & \bmzero & \bmzero & \bmzero & \bmoner
\end{array}$} &&\to \scalebox{0.6}{$\begin{array}{c>{\columncolor{red!20}}cccc}
\bmzero & \bmzero & \bmoner & \bmzero & \bmzero \\
\bmzero & \bmzero & \bmoner & \bmone & \bmone \\
\bmone & \bmzero & \bmzero & \bmoner & \bmone\\
\bmzero & \bmzero & \bmzero & \bmoner & \bmzero \\
\bmone & \bmone & \bmone & \bmzero & \bmoner \\
\bmone & \bmzero & \bmzero & \bmzero & \bmoner \\
\bmzero & \bmzero & \bmzero & \bmzero & \bmoner
\end{array}$} &&\to \scalebox{0.6}{$\begin{array}{cc>{\columncolor{red!20}}ccc}
\bmzero & \bmoner & \bmzero & \bmzero & \bmzero \\
\bmzero & \bmzero & \bmoner & \bmone & \bmone \\
\bmone & \bmzero & \bmzero & \bmoner & \bmone\\
\bmzero & \bmzero & \bmzero & \bmoner & \bmzero \\
\bmone & \bmone & \bmone & \bmzero & \bmoner \\
\bmone & \bmzero & \bmzero & \bmzero & \bmoner \\
\bmzero & \bmzero & \bmzero & \bmzero & \bmoner
\end{array}$}\to\\
\scalebox{0.6}{$\begin{array}{c>{\columncolor{red!20}}cccc}
\bmzero & \bmoner & \bmzero & \bmzero & \bmzero \\
\bmzero & \bmzero & \bmoner & \bmone & \bmone \\
\bmzero & \bmzero & \bmoner & \bmone & \bmone\\
\bmzero & \bmzero & \bmzero & \bmoner & \bmzero \\
\bmone & \bmone & \bmone & \bmzero & \bmoner \\
\bmone & \bmzero & \bmzero & \bmzero & \bmoner \\
\bmzero & \bmzero & \bmzero & \bmzero & \bmoner
\end{array}$}&&\to\scalebox{0.6}{$\begin{array}{>{\columncolor{red!20}}ccccc}
\bmzero & \bmoner & \bmzero & \bmzero & \bmzero \\
\bmzero & \bmoner & \bmone & \bmone & \bmzero \\
\bmzero & \bmzero & \bmoner & \bmone & \bmone\\
\bmzero & \bmzero & \bmzero & \bmoner & \bmzero \\
\bmone & \bmone & \bmone & \bmzero & \bmoner \\
\bmone & \bmzero & \bmzero & \bmzero & \bmoner \\
\bmzero & \bmzero & \bmzero & \bmzero & \bmoner
\end{array}$}&&\to \scalebox{0.6}{$\begin{array}{c>{\columncolor{red!20}}cccc}
\bmoner & \bmzero & \bmzero & \bmzero & \bmzero \\
\bmzero & \bmoner & \bmone & \bmone & \bmzero \\
\bmzero & \bmzero & \bmoner & \bmone & \bmone\\
\bmzero & \bmzero & \bmzero & \bmoner & \bmzero \\
\bmone & \bmone & \bmone & \bmzero & \bmoner \\
\bmone & \bmzero & \bmzero & \bmzero & \bmoner \\
\bmzero & \bmzero & \bmzero & \bmzero & \bmoner
\end{array}$}&&\to\scalebox{0.6}{$\begin{array}{ccccc}
\bmoner & \bmzero & \bmzero & \bmzero & \bmzero \\
\bmzero & \bmoner & \bmone & \bmone & \bmzero \\
\bmzero & \bmoner & \bmone & \bmone & \bmzero\\
\bmzero & \bmzero & \bmzero & \bmoner & \bmzero \\
\bmone & \bmone & \bmone & \bmzero & \bmoner \\
\bmone & \bmzero & \bmzero & \bmzero & \bmoner \\
\bmzero & \bmzero & \bmzero & \bmzero & \bmoner
\end{array}$}\makebox[\widthof{$\to$}][c]{}
\end{alignat*}

Thus, the rightmost equitable partition of $\uu=1331421$ is $\uu^*=1|33|\cdot|1|421$.
\label{ex:rightmost_eq}
\end{example}

\begin{example}
We illustrate~\Cref{map:leftmost} applied to the word $\uu=1342143$.  At each step, the rightmost column with less than its equitable filling is highlighted.

\begin{alignat*}{4}\scalebox{0.6}{$\begin{array}{ccc>{\columncolor{red!20}}cc}
\bmzero & \bmzero & \bmzero & \makebox[\widthof{$\bmone$}][c]{$\bmzero$} & \bmoner \\
\bmone & \bmone & \bmzero & \bmzero & \bmoner \\
\bmone & \bmone & \bmone & \bmzero & \bmoner\\
\bmone & \bmzero & \bmzero & \bmzero & \bmoner \\
\bmzero & \bmzero & \bmzero & \bmzero & \bmoner \\
\bmone & \bmone & \bmone & \bmzero & \bmoner \\
\bmone & \bmone & \bmzero & \bmzero & \bmoner
\end{array}$} &&\to \scalebox{0.6}{$\begin{array}{ccc>{\columncolor{red!20}}cc}
\bmzero & \bmzero & \bmzero & \bmoner & \bmzero \\
\bmone & \bmone & \bmzero & \bmzero & \bmoner \\
\bmone & \bmone & \bmone & \bmzero & \bmoner\\
\bmone & \bmzero & \bmzero & \bmzero & \bmoner \\
\bmzero & \bmzero & \bmzero & \bmzero & \bmoner \\
\bmone & \bmone & \bmone & \bmzero & \bmoner \\
\bmone & \bmone & \bmzero & \bmzero & \bmoner
\end{array}$} &&\to \scalebox{0.6}{$\begin{array}{ccc>{\columncolor{red!20}}cc}
\bmzero & \bmzero & \bmzero & \bmoner & \bmzero \\
\bmone & \bmzero & \bmzero & \bmoner & \bmone \\
\bmone & \bmone & \bmone & \bmzero & \bmoner\\
\bmone & \bmzero & \bmzero & \bmzero & \bmoner \\
\bmzero & \bmzero & \bmzero & \bmzero & \bmoner \\
\bmone & \bmone & \bmone & \bmzero & \bmoner \\
\bmone & \bmone & \bmzero & \bmzero & \bmoner
\end{array}$} &&\to \scalebox{0.6}{$\begin{array}{ccc>{\columncolor{red!20}}cc}
\bmzero & \bmzero & \bmzero & \bmoner & \bmzero \\
\bmone & \bmzero & \bmzero & \bmoner & \bmone \\
\bmone & \bmone & \bmzero & \bmoner & \bmone\\
\bmone & \bmzero & \bmzero & \bmzero & \bmoner \\
\bmzero & \bmzero & \bmzero & \bmzero & \bmoner \\
\bmone & \bmone & \bmone & \bmzero & \bmoner \\
\bmone & \bmone & \bmzero & \bmzero & \bmoner
\end{array}$} \to \\
 \scalebox{0.6}{$\begin{array}{cc>{\columncolor{red!20}}ccc}
\bmzero & \bmzero & \bmzero & \bmoner & \bmzero \\
\bmone & \bmzero & \bmzero & \bmoner & \bmone \\
\bmone & \bmone & \bmzero & \bmoner & \bmone\\
\bmzero & \bmzero & \bmzero & \bmoner & \bmone \\
\bmzero & \bmzero & \bmzero & \bmzero & \bmoner \\
\bmone & \bmone & \bmone & \bmzero & \bmoner \\
\bmone & \bmone & \bmzero & \bmzero & \bmoner
\end{array}$} &&\to \scalebox{0.6}{$\begin{array}{ccc>{\columncolor{red!20}}cc}
\bmzero & \bmzero & \bmoner & \bmzero & \bmzero \\
\bmone & \bmzero & \bmzero & \bmoner & \bmone \\
\bmone & \bmone & \bmzero & \bmoner & \bmone\\
\bmzero & \bmzero & \bmzero & \bmoner & \bmone \\
\bmzero & \bmzero & \bmzero & \bmzero & \bmoner \\
\bmone & \bmone & \bmone & \bmzero & \bmoner \\
\bmone & \bmone & \bmzero & \bmzero & \bmoner
\end{array}$} &&\to \scalebox{0.6}{$\begin{array}{cc>{\columncolor{red!20}}ccc}
\bmzero & \bmzero & \bmoner & \bmzero & \bmzero \\
\bmone & \bmzero & \bmzero & \bmoner & \bmone \\
\bmone & \bmone & \bmzero & \bmoner & \bmone\\
\bmzero & \bmzero & \bmzero & \bmoner & \bmone \\
\bmzero & \bmzero & \bmzero & \bmoner & \bmzero \\
\bmone & \bmone & \bmone & \bmzero & \bmoner \\
\bmone & \bmone & \bmzero & \bmzero & \bmoner
\end{array}$} &&\to \scalebox{0.6}{$\begin{array}{cc>{\columncolor{red!20}}ccc}
\bmzero & \bmzero & \bmoner & \bmzero & \bmzero \\
\bmzero & \bmzero & \bmoner & \bmone & \bmone \\
\bmone & \bmone & \bmzero & \bmoner & \bmone\\
\bmzero & \bmzero & \bmzero & \bmoner & \bmone \\
\bmzero & \bmzero & \bmzero & \bmoner & \bmzero \\
\bmone & \bmone & \bmone & \bmzero & \bmoner \\
\bmone & \bmone & \bmzero & \bmzero & \bmoner
\end{array}$}\to\\
\scalebox{0.6}{$\begin{array}{c>{\columncolor{red!20}}cccc}
\bmzero & \bmzero & \bmoner & \bmzero & \bmzero \\
\bmzero & \bmzero & \bmoner & \bmone & \bmone \\
\bmone & \bmzero & \bmoner & \bmone & \bmone\\
\bmzero & \bmzero & \bmzero & \bmoner & \bmone \\
\bmzero & \bmzero & \bmzero & \bmoner & \bmzero \\
\bmone & \bmone & \bmone & \bmzero & \bmoner \\
\bmone & \bmone & \bmzero & \bmzero & \bmoner
\end{array}$}&&\to\scalebox{0.6}{$\begin{array}{cc>{\columncolor{red!20}}ccc}
\bmzero & \bmoner & \bmzero & \bmzero & \bmzero \\
\bmzero & \bmzero & \bmoner & \bmone & \bmone \\
\bmone & \bmzero & \bmoner & \bmone & \bmone\\
\bmzero & \bmzero & \bmzero & \bmoner & \bmone \\
\bmzero & \bmzero & \bmzero & \bmoner & \bmzero \\
\bmone & \bmone & \bmone & \bmzero & \bmoner \\
\bmone & \bmone & \bmzero & \bmzero & \bmoner
\end{array}$}&&\to \scalebox{0.6}{$\begin{array}{c>{\columncolor{red!20}}cccc}
\bmzero & \bmoner & \bmzero & \bmzero & \bmzero \\
\bmzero & \bmzero & \bmoner & \bmone & \bmone \\
\bmone & \bmzero & \bmoner & \bmone & \bmone\\
\bmzero & \bmzero & \bmoner & \bmone & \bmzero \\
\bmzero & \bmzero & \bmzero & \bmoner & \bmzero \\
\bmone & \bmone & \bmone & \bmzero & \bmoner \\
\bmone & \bmone & \bmzero & \bmzero & \bmoner
\end{array}$}&&\to \scalebox{0.6}{$\begin{array}{ccccc}
\bmzero & \bmoner & \bmzero & \bmzero & \bmzero \\
\bmzero & \bmoner & \bmone & \bmone & \bmzero \\
\bmone & \bmzero & \bmoner & \bmone & \bmone\\
\bmzero & \bmzero & \bmoner & \bmone & \bmzero \\
\bmzero & \bmzero & \bmzero & \bmoner & \bmzero \\
\bmone & \bmone & \bmone & \bmzero & \bmoner \\
\bmone & \bmone & \bmzero & \bmzero & \bmoner
\end{array}$}\makebox[\widthof{$\to$}][c]{}
\end{alignat*}

Thus, the rightmost equitable partition of $\uu$ is $\uu^*=\cdot|13|42|1|43$.
\label{ex:rightmost_eq2}
\end{example}

We begin with a lemma about the filling of column $0$ during the execution of~\Cref{map:leftmost}.

\begin{lemma} \label{lem:more-than-eq}
While running~\Cref{map:leftmost}, before we have reached an equitable
filling, column $0$ will always be more than equitably filled.
\end{lemma}

\begin{proof}
The set of columns which are more than equitably filled can only decrease over time, since we never make a move which increases the number of copies of $\bmone$ in a column above the equitable amount.  Thus, if a column is more than equitably filled at a certain point, it has been more than equitably filled since the beginning.

Consider some $j>0$, and suppose that at some point in the execution of~\Cref{map:leftmost}, column $j$ is more than equitably filled.  By what we have just argued, column $j$ has been more than equitably filled from the beginning. Thus, we never pushed any letters from the $(j-1)$st block into the $j$th block, and we never moved any letters into any block to the left of the $j$th block (since any such letter starts in block $0$ and so would have had to pass through the $j$th block).  So the reason that the $j$th column is more than equitably filled is accounted for entirely by wrap-around from letters which are in blocks to its right.  Every column to the left of $j$, and column 0, will also receive copies of $\bmone$ from each of the letters which contribute to the $j$th column.  It follows that they are all more than equitably filled.

Thus, at each step of the algorithm, if the filling is not yet equitable, the columns which are more than equitably filled consist of an initial (left-aligned) subsequence of the columns, possibly empty, together with column 0. 
\end{proof}

\begin{theorem}
\label{thm:unique_rightmost}
Any $\uu \in \Am$ admits a unique rightmost equitable partition $\invf(\uu)$.
\end{theorem}


\begin{proof}
We claim that~\Cref{map:leftmost} constructs the unique rightmost equitable partition.  We first show that~\Cref{map:leftmost} does not attempt any illegal moves, and so returns an equitable partition.  Let $j$ be the rightmost column which is less than equitably filled.
  \Cref{map:leftmost} might fail in two ways:
\begin{enumerate}
\item Suppose $j=0$.  If column $0$ is less than equitably filled, then \Cref{map:leftmost} would attempt to delete the leftmost letter of the nonexistent block $\bb_{-1}$.  But this case cannot arise by~\Cref{lem:more-than-eq}.
\item Now suppose $j>0$.  \Cref{map:leftmost} might try to delete the leftmost entry of an \emph{empty} block $\bb_{j-1}$ when column $j>0$ is less than equitably filled.

If $j>1$, since column $j-1$ is to the right of $j$, it must be at least equitably filled.  Since the number of copies of $\bmone$ in an equitable filling of column $j$ is at most the number in an equitable filling of column $j-1$, and column $j-1$ is at least equitably filled while $j$ is not, there must be a letter in block $j-1$.

If $j=1$, the above argument fails because the number of copies of $\bmone$ in an equitable filling of column 1 may be greater (by one)
than the number in an equitable filling of column 0.  But by \Cref{lem:more-than-eq}, we know that column 0 is strictly more than equitably filled, and so if column 1 is less than equitably filled, there must be a letter in block $0$.
\end{enumerate}

At each step letters are moved further to the left, so we never obtain the same partition twice.    Since there are only a finite number of partitioned words,~\Cref{map:leftmost} must eventually terminate.

We now show that \Cref{map:leftmost} outputs the unique rightmost equitable partition $\uu^*$.  For each step of \Cref{map:leftmost}, record the pair $(j,i)$, where $j$ is the rightmost column that is less than equitably filled, and $i$ is the index of the letter $\uu_i$ that we move from $\uu_{j-1}^*$
to $\uu_{j}^*$.  Consider another equitable partition $\vv^*$ of $\uu$ and find the first recorded pair $(j,i)$ such that $\uu_i$ is in $\vv_{j-1}^*$.  If there is no such pair, then each letter of $\uu$ is in a block at least as far to the left in $\vv^*$
as in $\uu^*$ and we are done.

Otherwise, write $\uu^\partial$ for
the partition produced by running~\Cref{map:leftmost} up to (but not including) the step that produced the pair $(j,i)$.  By the choice of the pair $(j,i)$, every letter of $\uu$ is at least as far to the left in $\vv^*$ as in $\uu^\partial$, since the step at which we recorded $(j,i)$ is precisely the first step in \Cref{map:leftmost} in which we moved a letter to the left of its position in $\vv^*$.  The partition $\uu^\partial$ is non-equitable,
and $j$ is its rightmost column that is less than equitably filled.  Since every letter in a block to the left of $j$ that does not contribute a $\bmone$ to column $j$ for $\uu^\partial$ is at least as far to the left in $\vv^*$ as in $\uu^\partial$, it does not contribute
a $\bmone$ to column $j$ in $\vv^*$ either.  Every letter in a block to the right of $j$ is also at least as far to the left in $\vv^*$ as in $\uu^\partial$, and at the same time, since
the
initial letter in $\uu_{j-1}^\partial$ is by assumption in $\vv^*_{j-1}$, no letter
that does not contribute a $\bmone$ to column $j$ for $\uu^\partial$ can have moved to the left of
$\vv^*_{j-1}$ to contribute a $\bmone$ to column $\vv^*_j$.  Thus, $\vv^*$ has at most as many copies of $\bmone$ in
column $j$ as $\uu^\partial$ does, which is too few to be equitable, contradicting our assumption
that $\vv^*$ was equitable.
\end{proof}

\subsection{Balanced Block Suffixes and the Successful Partition}  In this section, we give some structural definitions and results on the poset of equitable partitions $\equ(\uu)$, and give alternative characterizations of the rightmost and successful partitions.  Using these characterizations, we prove that the unique rightmost equitable partition is the unique successful partition.

\begin{definition}
Fix $\uu \in \Am$ and $\uu^* \in \equ(\uu)$.  A \defn{balanced block suffix} is a partitioned subword $\us^*$ of $\uu^*$ whose intersection with each block of $\uu^*$ is a suffix of that block, and that contributes the same number of copies of $\bmone$ to each column of $M_{\uu^*}$.  In particular, $\us^*$ is an equitable partition of $\us$ and $|\us|_m=0$.

A \defn{left balanced block suffix} has the additional constraint that it has empty rightmost block: $\us^*_0=\emptyset$. A \defn{minimal balanced block suffix} is a balanced block suffix whose intersection with all other balanced block suffixes is itself or empty.
\label{def:block_suffix}
\end{definition}

All minimal balanced block suffixes for the equitable partitions of $1331421$ and $1342143$ are shown in~\Cref{ex:componentwise}.  In~\Cref{ex:componentwise}, we differentiate different minimal balanced block suffixes by underlining one and wavy underlining the other.  Two particular equitable partitions are examined in more detail below in~\Cref{ex:min_block_suffix}.

\begin{example}
\label{ex:min_block_suffix}
The equitable partitions $13|31|4|2|1$ and $1|34|21|4|3$ each have three balanced block suffixes:
\setlength{\arraycolsep}{3pt}
\[\begin{array}{cccc|cc|c|c|c}
\us^*&=&\textcolor{lightgray}{1}&3&\textcolor{lightgray}{3}&\textcolor{lightgray}{1}&\textcolor{lightgray}{4}&2&\textcolor{lightgray}{1}\\
\ut^*&=&\textcolor{lightgray}{1}&\textcolor{lightgray}{3}&\textcolor{lightgray}{3}&1&4&\textcolor{lightgray}{2}&\textcolor{lightgray}{1}\\
\us^* \cup \ut^*&=&\textcolor{lightgray}{1}&3&\textcolor{lightgray}{3}&1&4&2&\textcolor{lightgray}{1}\\
\end{array} \text{ and } \begin{array}{ccc|cc|cc|c|c}
\us^*&=&\textcolor{lightgray}{1}&\textcolor{lightgray}{3}&\textcolor{lightgray}{4}&\textcolor{lightgray}{2}&1&4&\textcolor{lightgray}{3}\\
\ut^*&=&1&\textcolor{lightgray}{3}&4&\textcolor{lightgray}{2}&\textcolor{lightgray}{1}&\textcolor{lightgray}{4}&\textcolor{lightgray}{3}\\
\us^* \cup \ut^*&=&1&\textcolor{lightgray}{3}&4&\textcolor{lightgray}{2}&1&4&\textcolor{lightgray}{3}\\
\end{array}.\]


For each equitable partition, all three are left balanced block suffixes, but only the first two are minimal. 
\end{example}

The following lemma establishes the existence of a \defn{maximum left balanced block suffix}---that is, a left balanced block suffix containing all left balanced block suffixes.

\begin{lemma}
\label{lem:notvisited}
For any $\uu^* \in \equ(\uu)$, $\uu^*$ is not successful if and only if there exists a nonempty left balanced block suffix.  Moreover, in all cases (successful or not), the letters of $\uu^*$ not removed by~\Cref{map:backwards} form the maximum left balanced block suffix $\us^*$ of $\uu^*$ (possibly empty).
\end{lemma}

\begin{proof}
\Cref{map:backwards} terminates unsuccessfully when it tries to remove a letter from an empty block of $\uu^*$.  This block corresponds to a column in $M_{\uu^*}$ whose remaining copies of $\bmone$ all came from other blocks.  Since every time~\Cref{map:backwards} removes a letter from $\uu^*$ the new partitioned word is still equitable, it must remove a letter from a block corresponding to the leftmost column with the most copies of $\bmone$, or a letter from a block corresponding to the rightmost column (if all columns have the same number of copies of $\bmone$).  Then \Cref{map:backwards} does not succeed only if it was trying to remove a letter from the empty rightmost block, so that the copies of $\bmone$ corresponding to the remaining letters in $\uu^*$ are equally distributed among the columns, and there are no remaining letters in the rightmost block.  This is the condition to be a left balanced block suffix.


Now let $\us^*$ be any left balanced block suffix. Since $\us^*$ is equitable with $|\us|_m = 0$, removing $\us^*$ from $\uu^*$ leaves an equitable partition $\uu^*{-}\us^*$. Since each block of $\us^*$ is a suffix of the corresponding block of $\uu^*$, \Cref{map:backwards} will remove the same letters from $\uu^*{-}\us^*$ that it removed from $\uu^*$. In particular, \Cref{map:backwards} with input $\uu^*{-}\us^*$ still fails when it reaches an empty rightmost block. Since $\us^*_0$ is empty by assumption, we conclude that no letters in $\us^*$ were removed when \Cref{map:backwards} was run on $\uu^*$. Thus $\us^*$ is contained in the left balanced block suffix of $\uu^*$ consisting of all letters not removed by \Cref{map:backwards}.
\end{proof}

\begin{lemma}
\label{lem:cascade}
Let $\uu_i$ be a letter of $\uu^*$ that is rightmost in its block, but not as far to the right as it is in $\invf(\uu)$.  Then $\uu_i$ is part of a minimal left balanced block suffix $\us^*$.
\end{lemma}
\begin{proof}
Consider the sequence of forced moves
that follow once we move $\uu_i$ one block to the right: specifically, moving it to the right means that some
column now has too many copies of $\bmone$, which forces a letter to move out of that block, which means that another column now has too many copies of $\bmone$, and so on.  We refer to this sequence as the
\defn{cascade} of forced moves.  The cascade never tries to move a letter out of an empty block because
there is necessarily an element available whenever needed.  It also never tries to move a letter out of the rightmost block because if $\uu_i$ can be further to the right, then so can any other element in the cascade.
The cascade terminates since there are only a finite number of moves that can be applied before a box will be moved back into the column from which $\uu_i$ was initially removed, which will terminate the cascade; the resulting partition is necessarily again equitable.  This is illustrated in~\Cref{ex:cascade}.

We now show that the cascade never moves any letter more than once.  Suppose that up to a certain step in the cascade, we have not shifted any letter more than once, but we now move a letter $\uu_j$ for the second time.
At this point, let $\uu^{\partial}$ be the partition and let $\block(\uu^{\partial},j)=k$. Suppose that $\uu^*_k$ has $\ell$ letters, so that all $\ell$ letters originally in block $\uu^*_k$ have already been shifted.  Then the block to which our first letter $\uu_i$ belonged---$\block(\uu^*,i)$---is not equal to $k$, because the cascade would have concluded upon arriving back in block $\block(\uu^*,i)$.  Also, since it is impossible to shift letters from the rightmost block, we must have $k\neq 0$.  In order to shift a letter from block $k$, we must have just finished shifting a letter whose corresponding row in $M^{\uu^{\partial}}$ has rightmost $\bmone$ in column $k + 1$. Since $\uu^{*}$ is equitable, we know that in $\uu^*$, the number of letters whose rightmost $\bmone$ is in column
$k + 1$ is equal to or one less than the number of letters in $\uu^*_k$.  By our assumption that we have moved letters at most once in our cascade from $\uu^{*}$ to $\uu^{\partial}$, it follows that we have
previously moved at most $\ell-1$ letters out of block $\uu^*_k$---one fewer than the number which
it started with, contrary to our assumption that at this time, we had already moved all of them.

Let $\us^*$ be the partitioned word consisting of the letters moved in the cascade.  Since each of the letters in $\us^*$ is moved exactly once and $\uu_i$ is among them, $\us^*$ is a nonempty left balanced block suffix.  It is also clear that $\us^*$ is minimal.
\end{proof}

\begin{example}
We illustrate the cascade of forced moves in the proof of \Cref{lem:cascade}.  Consider the equitable partition $\uu^* = 1\underline{1}| 3\underline{2}|\underline{3}| 2\underline{4}|\cdot$ with $m=5$, where a minimal left balanced block suffix has been underlined.  This balanced block suffix contributes two copies of $\bmone$ to each column of $M_{\uu^*}$.  Suppose we move the underlined $\underline{2}$ from block 3 to block 2 to form
$1\underline{1}| 3|\underline{2}\underline{3}| 2\underline{4}|\cdot$.  Then column 1 has too many copies of $\bmone$, and so we adjust to $1\underline{1}| 3|\underline{2}\underline{3}| 2|\underline{4}$, then to $1\underline{1}| 3|\underline{2}| \underline{3}2|\underline{4}$ and finally arrive at $1| \underline{1}3|\underline{2}| \underline{3}2|\underline{4}$, which is once again an equitable partition.  This process is illustrated below, where the column with too many copies of $\bmone$ is highlighted, and the copies of $\bmone$ in the row corresponding to the letter that was forced to move are colored.

\begin{alignat*}{5}\scalebox{0.6}{$\begin{array}{ccccc}
\bmone & \bmzero & \bmzero & \bmzero & \bmzero \\
\bmoner & \bmzero & \bmzero & \bmzero & \bmzero \\
\bmzero & \bmone & \bmone & \bmone & \bmzero\\
\bmzero & \bmoner & \bmoner & \bmzero & \bmzero \\
\bmzero & \bmzero & \bmoner & \bmoner & \bmoner \\
\bmzero & \bmzero & \bmzero & \bmone & \bmone \\
\bmoner & \bmoner & \bmzero & \bmoner & \bmoner
\end{array}$} &&\to \scalebox{0.6}{$\begin{array}{ccc>{\columncolor{red!20}}cc}
\bmone & \bmzero & \bmzero & \bmzero & \bmzero \\
\bmone & \bmzero & \bmzero & \bmzero & \bmzero \\
\bmzero & \bmone & \bmone & \bmone & \bmzero\\
\bmzero & \bmzero & \bmoner & \bmoner & \bmzero \\
\bmzero & \bmzero & \bmone & \bmone & \bmone \\
\bmzero & \bmzero & \bmzero & \bmone & \bmone \\
\bmone & \bmone & \bmzero & \bmone & \bmone
\end{array}$} &&\to \scalebox{0.6}{$\begin{array}{cc>{\columncolor{red!20}}ccc}
\bmone & \bmzero & \bmzero & \bmzero & \bmzero \\
\bmone & \bmzero & \bmzero & \bmzero & \bmzero \\
\bmzero & \bmone & \bmone & \bmone & \bmzero\\
\bmzero & \bmzero & \bmone & \bmone & \bmzero \\
\bmzero & \bmzero & \bmone & \bmone & \bmone \\
\bmzero & \bmzero & \bmzero & \bmone & \bmone \\
\bmoner & \bmoner & \bmoner & \bmzero & \bmoner
\end{array}$} &&\to \scalebox{0.6}{$\begin{array}{>{\columncolor{red!20}}ccccc}
\bmone & \bmzero & \bmzero & \bmzero & \bmzero \\
\bmone & \bmzero & \bmzero & \bmzero & \bmzero \\
\bmzero & \bmone & \bmone & \bmone & \bmzero\\
\bmzero & \bmzero & \bmone & \bmone & \bmzero \\
\bmoner & \bmzero & \bmzero & \bmoner & \bmoner \\
\bmzero & \bmzero & \bmzero & \bmone & \bmone \\
\bmone & \bmone & \bmone & \bmzero & \bmone
\end{array}$} &&\to \scalebox{0.6}{$\begin{array}{ccccc}
\bmone & \bmzero & \bmzero & \bmzero & \bmzero \\
\bmzero & \bmoner & \bmzero & \bmzero & \bmzero \\
\bmzero & \bmone & \bmone & \bmone & \bmzero\\
\bmzero & \bmzero & \bmone & \bmone & \bmzero \\
\bmone & \bmzero & \bmzero & \bmone & \bmone \\
\bmzero & \bmzero & \bmzero & \bmone & \bmone \\
\bmone & \bmone & \bmone & \bmzero & \bmone
\end{array}$}
\end{alignat*}

\label{ex:cascade}
\end{example}

The following lemma is the counterpart of the characterization in~\Cref{lem:notvisited} for the rightmost equitable partition.

\begin{lemma}
For any $\uu^* \in \equ(\uu)$, $\uu^*$ has a nonempty left balanced block suffix if and only if $\uu^* \neq \invf(\uu)$.
\label{thm:unique_succ}
\end{lemma}

\begin{proof}
If $\uu^*$ has a nonempty left balanced block suffix, let $\us^*$ be the partitioned word containing those letters that remain in $\uu^*$ after applying \Cref{map:backwards}.  By \Cref{lem:notvisited}, these form a left balanced block suffix.  Shifting the letters in $\us^*$ one block to the right in $\uu^*$---while preserving those letters in $\uu^*$ that do not appear in $\us^*$---we obtain a new equitable partition for $\uu$ that is greater than $\uu^*$ in the poset $\equ(\uu)$.

If $\uu^* \neq \invf(\uu)$, then there is some letter $\uu_i$ of $\uu^*$ that is not as far right as it is in $\invf(\uu)$; we can take $\uu_i$ to be rightmost in its block.   Then $\uu^*$ has a nonempty left balanced block suffix by \Cref{lem:cascade}.
\end{proof}


\begin{theorem} 
For any partition $\uu^*$ of $\uu$, $\uu^*$ is successful if and only if $\uu^*$ is the rightmost equitable partition.
\label{thm:right_eq_succ}
\end{theorem}
\begin{proof}
By~\Cref{lem:succ_equ}, any successful partition is necessarily equitable.  We conclude using~\Cref{lem:notvisited} and~\Cref{thm:unique_succ} that $\uu^*$ is successful if and only if it has no nonempty left balanced block suffix if and only if it is \emph{the} rightmost equitable partition (\Cref{thm:unique_rightmost}).
\end{proof}

\subsection{The Modular Sweep Map}
\label{sec:sweep}

We can now prove our main theorem.

{
\renewcommand{\thetheorem}{\ref{thm:main_thm}}
\begin{theorem}
The modular sweep map is a bijection $\Am \to \Am$.
\end{theorem}
\addtocounter{theorem}{-1}
}

\begin{proof}
Recall that the modular sweep map may be written as the composition \[\sw = \f \circ \swp: \Am \to \Am.\]  Since~\Cref{sec:inv_modular_presweep} inverts $\swp$ and~\Cref{map:leftmost} inverts $\f$ (using~\Cref{thm:right_eq_succ}), we conclude that the modular sweep map is invertible and is given by the composition of~\Cref{map:leftmost} with~\Cref{map:backwards}.
\end{proof}

\section{The Distributive Lattice of Equitable Partitions}\label{sec:eq_lattice}

In this section, motivated by an analogy to stable marriages, we show that the poset of equitable partitions is a distributive lattice.

\subsection{The Leftmost Equitable Partition}\label{sec:lefmost}  Having defined the \emph{rightmost} equitable partition, it is natural to ask for the \emph{leftmost} equitable partition.  As it turns out, there is exactly one leftmost equitable partition---it is the minimum element of $\equ(\uu)$.

\begin{definition}[{\cite[Definition 8.3]{thomas2014cyclic}}]\label{def:leftmost}
A \defn{leftmost equitable partition} is a minimal element of $\equ(\uu)$.
\end{definition}

That is, a leftmost equitable partition is an equitable partition $\uu^*$ of $\uu$ such that any other equitable partition $\vv^{*}$ of $\uu$ has $\block(\vv^{*},i)\leq \block(\uu^*,i)$ for all $i$.  \Cref{ex:componentwise} indicates the leftmost equitable partition in the poset of all equitable partitions of the words $\uu=1331421$ and $\uu=1342143$.  As we will show,~\Cref{map:rightmost} explicitly constructs the leftmost equitable partition.

\begin{algorithm}[ht]
\KwIn{A word $\uu \in \Am$.}
\KwOut{The leftmost equitable partition $\uu^* \in \Am^*$.}
Set $\uu^*=\uu|\cdot|\cdot|\cdots|\cdot$\;
\While{$\uu^*$ is not an equitable partition}{
		Let $j$ be the leftmost column of $M^{\uu^*}$ that is more than equitably filled\;
		Delete the rightmost letter of $\bb_{j}$ and prepend it to $\bb_{j-1}$\;
}
Return $\uu^*$\;
\caption{$\rightmost: \Am \to \Am^*$.}
\label{map:rightmost}
\end{algorithm}

\begin{example}
We illustrate~\Cref{map:rightmost} applied to the word $\uu=1331421$.  At each step, the leftmost column with more than its equitable filling is highlighted.

\begin{alignat*}{4}\scalebox{0.6}{$\begin{array}{>{\columncolor{red!20}}ccccc}
\bmoner & \bmzero & \bmzero & \bmzero & \makebox[\widthof{$\bmzero$}][c]{$\bmzero$} \\
\bmoner & \bmone & \bmone & \bmzero & \bmzero \\
\bmoner & \bmone & \bmone & \bmzero & \bmzero\\
\bmoner & \bmzero & \bmzero & \bmzero & \bmzero \\
\bmoner & \bmone & \bmone & \bmone & \bmzero \\
\bmoner & \bmone & \bmzero & \bmzero & \bmzero \\
\bmoner & \bmzero & \bmzero & \bmzero & \bmzero
\end{array}$} &&\to \scalebox{0.6}{$\begin{array}{>{\columncolor{red!20}}ccccc}
\bmoner & \bmzero & \bmzero & \bmzero & \makebox[\widthof{$\bmzero$}][c]{$\bmzero$} \\
\bmoner & \bmone & \bmone & \bmzero & \bmzero \\
\bmoner & \bmone & \bmone & \bmzero & \bmzero\\
\bmoner & \bmzero & \bmzero & \bmzero & \bmzero \\
\bmoner & \bmone & \bmone & \bmone & \bmzero \\
\bmoner & \bmone & \bmzero & \bmzero & \bmzero \\
\bmzero & \bmoner & \bmzero & \bmzero & \bmzero
\end{array}$} &&\to \scalebox{0.6}{$\begin{array}{>{\columncolor{red!20}}ccccc}
\bmoner & \bmzero & \bmzero & \bmzero & \makebox[\widthof{$\bmzero$}][c]{$\bmzero$} \\
\bmoner & \bmone & \bmone & \bmzero & \bmzero \\
\bmoner & \bmone & \bmone & \bmzero & \bmzero\\
\bmoner & \bmzero & \bmzero & \bmzero & \bmzero \\
\bmoner & \bmone & \bmone & \bmone & \bmzero \\
\bmzero & \bmoner & \bmone & \bmzero & \bmzero \\
\bmzero & \bmoner & \bmzero & \bmzero & \bmzero
\end{array}$} &&\to \scalebox{0.6}{$\begin{array}{>{\columncolor{red!20}}ccccc}
\bmoner & \bmzero & \bmzero & \bmzero & \bmzero \\
\bmoner & \bmone & \bmone & \bmzero & \bmzero \\
\bmoner & \bmone & \bmone & \bmzero & \bmzero\\
\bmoner & \bmzero & \bmzero & \bmzero & \bmzero \\
\bmzero & \bmoner & \bmone & \bmone & \bmone \\
\bmzero & \bmoner & \bmone & \bmzero & \bmzero \\
\bmzero & \bmoner & \bmzero & \bmzero & \bmzero
\end{array}$} \to \\
\scalebox{0.6}{$\begin{array}{c>{\columncolor{red!20}}cccc}
\bmoner & \bmzero & \bmzero & \bmzero & \bmzero \\
\bmoner & \bmone & \bmone & \bmzero & \bmzero \\
\bmoner & \bmone & \bmone & \bmzero & \bmzero\\
\bmzero & \bmoner & \bmzero & \bmzero & \bmzero \\
\bmzero & \bmoner & \bmone & \bmone & \bmone \\
\bmzero & \bmoner & \bmone & \bmzero & \bmzero \\
\bmzero & \bmoner & \bmzero & \bmzero & \bmzero
\end{array}$} &&\to \scalebox{0.6}{$\begin{array}{c>{\columncolor{red!20}}cccc}
\bmoner & \bmzero & \bmzero & \bmzero & \bmzero \\
\bmoner & \bmone & \bmone & \bmzero & \bmzero \\
\bmoner & \bmone & \bmone & \bmzero & \bmzero\\
\bmzero & \bmoner & \bmzero & \bmzero & \bmzero \\
\bmzero & \bmoner & \bmone & \bmone & \bmone \\
\bmzero & \bmoner & \bmone & \bmzero & \bmzero \\
\bmzero & \bmzero & \bmoner & \bmzero & \bmzero
\end{array}$} &&\to \scalebox{0.6}{$\begin{array}{c>{\columncolor{red!20}}cccc}
\bmoner & \bmzero & \bmzero & \bmzero & \bmzero \\
\bmoner & \bmone & \bmone & \bmzero & \bmzero \\
\bmoner & \bmone & \bmone & \bmzero & \bmzero\\
\bmzero & \bmoner & \bmzero & \bmzero & \bmzero \\
\bmzero & \bmoner & \bmone & \bmone & \bmone \\
\bmzero & \bmzero & \bmoner & \bmone & \bmzero \\
\bmzero & \bmzero & \bmoner & \bmzero & \bmzero
\end{array}$} &&\to \scalebox{0.6}{$\begin{array}{>{\columncolor{red!20}}ccccc}
\bmoner & \bmzero & \bmzero & \bmzero & \bmzero \\
\bmoner & \bmone & \bmone & \bmzero & \bmzero \\
\bmoner & \bmone & \bmone & \bmzero & \bmzero\\
\bmzero & \bmoner & \bmzero & \bmzero & \bmzero \\
\bmone & \bmzero & \bmoner & \bmone & \bmone \\
\bmzero & \bmzero & \bmoner & \bmone & \bmzero \\
\bmzero & \bmzero & \bmoner & \bmzero & \bmzero
\end{array}$}\to\\
\scalebox{0.6}{$\begin{array}{cc>{\columncolor{red!20}}ccc}
\bmoner & \bmzero & \bmzero & \bmzero & \bmzero \\
\bmoner & \bmone & \bmone & \bmzero & \bmzero \\
\bmzero & \bmoner & \bmone & \bmone & \bmzero\\
\bmzero & \bmoner & \bmzero & \bmzero & \bmzero \\
\bmone & \bmzero & \bmoner & \bmone & \bmone \\
\bmzero & \bmzero & \bmoner & \bmone & \bmzero \\
\bmzero & \bmzero & \bmoner & \bmzero & \bmzero
\end{array}$}&&\to\scalebox{0.6}{$\begin{array}{cc>{\columncolor{red!20}}ccc}
\bmoner & \bmzero & \bmzero & \bmzero & \bmzero \\
\bmoner & \bmone & \bmone & \bmzero & \bmzero \\
\bmzero & \bmoner & \bmone & \bmone & \bmzero\\
\bmzero & \bmoner & \bmzero & \bmzero & \bmzero \\
\bmone & \bmzero & \bmoner & \bmone & \bmone \\
\bmzero & \bmzero & \bmoner & \bmone & \bmzero \\
\bmzero & \bmzero & \bmzero & \bmoner & \bmzero
\end{array}$}&&\to \scalebox{0.6}{$\begin{array}{ccc>{\columncolor{red!20}}cc}
\bmoner & \bmzero & \bmzero & \bmzero & \bmzero \\
\bmoner & \bmone & \bmone & \bmzero & \bmzero \\
\bmzero & \bmoner & \bmone & \bmone & \bmzero\\
\bmzero & \bmoner & \bmzero & \bmzero & \bmzero \\
\bmone & \bmzero & \bmoner & \bmone & \bmone \\
\bmzero & \bmzero & \bmzero & \bmoner & \bmone \\
\bmzero & \bmzero & \bmzero & \bmoner & \bmzero
\end{array}$}&&\to\scalebox{0.6}{$\begin{array}{ccccc}
\bmoner & \bmzero & \bmzero & \bmzero & \bmzero \\
\bmoner & \bmone & \bmone & \bmzero & \bmzero \\
\bmzero & \bmoner & \bmone & \bmone & \bmzero\\
\bmzero & \bmoner & \bmzero & \bmzero & \bmzero \\
\bmone & \bmzero & \bmoner & \bmone & \bmone \\
\bmzero & \bmzero & \bmzero & \bmoner & \bmone \\
\bmzero & \bmzero & \bmzero & \bmzero & \bmoner
\end{array}$}\makebox[\widthof{$\to$}][c]{}
\end{alignat*}

Thus, the leftmost equitable partition of $\uu$ is $\uu^*=13|31|4|2|1$.
\label{ex:rightmost_example}
\end{example}

\begin{example}
We illustrate~\Cref{map:rightmost} applied to the word $\uu=1342143$.  At each step, the leftmost column with more than its equitable filling is highlighted.

\begin{alignat*}{5}\scalebox{0.6}{$\begin{array}{>{\columncolor{red!20}}ccccc}
\bmoner & \bmzero & \bmzero & \bmzero & \makebox[\widthof{$\bmzero$}][c]{$\bmzero$} \\
\bmoner & \bmone & \bmone & \bmzero & \bmzero \\
\bmoner & \bmone & \bmone & \bmone & \bmzero\\
\bmoner & \bmone & \bmzero & \bmzero & \bmzero \\
\bmoner & \bmzero & \bmzero & \bmzero & \bmzero \\
\bmoner & \bmone & \bmone & \bmone & \bmzero \\
\bmoner & \bmone & \bmone & \bmzero & \bmzero
\end{array}$} &&\to \scalebox{0.6}{$\begin{array}{>{\columncolor{red!20}}ccccc}
\bmoner & \bmzero & \bmzero & \bmzero & \makebox[\widthof{$\bmzero$}][c]{$\bmzero$} \\
\bmoner & \bmone & \bmone & \bmzero & \bmzero \\
\bmoner & \bmone & \bmone & \bmone & \bmzero\\
\bmoner & \bmone & \bmzero & \bmzero & \bmzero \\
\bmoner & \bmzero & \bmzero & \bmzero & \bmzero \\
\bmoner & \bmone & \bmone & \bmone & \bmzero \\
\bmzero & \bmoner & \bmone & \bmone & \bmzero
\end{array}$} &&\to \scalebox{0.6}{$\begin{array}{>{\columncolor{red!20}}ccccc}
\bmoner & \bmzero & \bmzero & \bmzero & \makebox[\widthof{$\bmzero$}][c]{$\bmzero$} \\
\bmoner & \bmone & \bmone & \bmzero & \bmzero \\
\bmoner & \bmone & \bmone & \bmone & \bmzero\\
\bmoner & \bmone & \bmzero & \bmzero & \bmzero \\
\bmoner & \bmzero & \bmzero & \bmzero & \bmzero \\
\bmzero & \bmoner & \bmone & \bmone & \bmone \\
\bmzero & \bmoner & \bmone & \bmone & \bmzero
\end{array}$} &&\to \scalebox{0.6}{$\begin{array}{>{\columncolor{red!20}}ccccc}
\bmoner & \bmzero & \bmzero & \bmzero & \makebox[\widthof{$\bmzero$}][c]{$\bmzero$} \\
\bmoner & \bmone & \bmone & \bmzero & \bmzero \\
\bmoner & \bmone & \bmone & \bmone & \bmzero\\
\bmoner & \bmone & \bmzero & \bmzero & \bmzero \\
\bmzero & \bmoner & \bmzero & \bmzero & \bmzero \\
\bmzero & \bmoner & \bmone & \bmone & \bmone \\
\bmzero & \bmoner & \bmone & \bmone & \bmzero
\end{array}$} &&\to \scalebox{0.6}{$\begin{array}{c>{\columncolor{red!20}}cccc}
\bmoner & \bmzero & \bmzero & \bmzero & \makebox[\widthof{$\bmzero$}][c]{$\bmzero$} \\
\bmoner & \bmone & \bmone & \bmzero & \bmzero \\
\bmoner & \bmone & \bmone & \bmone & \bmzero\\
\bmzero & \bmoner & \bmone & \bmzero & \bmzero \\
\bmzero & \bmoner & \bmzero & \bmzero & \bmzero \\
\bmzero & \bmoner & \bmone & \bmone & \bmone \\
\bmzero & \bmoner & \bmone & \bmone & \bmzero
\end{array}$} \to \\
 \scalebox{0.6}{$\begin{array}{c>{\columncolor{red!20}}cccc}
\bmoner & \bmzero & \bmzero & \bmzero & \makebox[\widthof{$\bmzero$}][c]{$\bmzero$} \\
\bmoner & \bmone & \bmone & \bmzero & \bmzero \\
\bmoner & \bmone & \bmone & \bmone & \bmzero\\
\bmzero & \bmoner & \bmone & \bmzero & \bmzero \\
\bmzero & \bmoner & \bmzero & \bmzero & \bmzero \\
\bmzero & \bmoner & \bmone & \bmone & \bmone \\
\bmzero & \bmzero & \bmoner & \bmone & \bmone
\end{array}$} &&\to \scalebox{0.6}{$\begin{array}{>{\columncolor{red!20}}ccccc}
\bmoner & \bmzero & \bmzero & \bmzero & \makebox[\widthof{$\bmzero$}][c]{$\bmzero$} \\
\bmoner & \bmone & \bmone & \bmzero & \bmzero \\
\bmoner & \bmone & \bmone & \bmone & \bmzero\\
\bmzero & \bmoner & \bmone & \bmzero & \bmzero \\
\bmzero & \bmoner & \bmzero & \bmzero & \bmzero \\
\bmone & \bmzero & \bmoner & \bmone & \bmone \\
\bmzero & \bmzero & \bmoner & \bmone & \bmone
\end{array}$} &&\to \scalebox{0.6}{$\begin{array}{cc>{\columncolor{red!20}}ccc}
\bmoner & \bmzero & \bmzero & \bmzero & \makebox[\widthof{$\bmzero$}][c]{$\bmzero$} \\
\bmoner & \bmone & \bmone & \bmzero & \bmzero \\
\bmzero & \bmoner & \bmone & \bmone & \bmone\\
\bmzero & \bmoner & \bmone & \bmzero & \bmzero \\
\bmzero & \bmoner & \bmzero & \bmzero & \bmzero \\
\bmone & \bmzero & \bmoner & \bmone & \bmone \\
\bmzero & \bmzero & \bmoner & \bmone & \bmone
\end{array}$} &&\to
 \scalebox{0.6}{$\begin{array}{>{\columncolor{red!20}}ccccc}
\bmoner & \bmzero & \bmzero & \bmzero & \makebox[\widthof{$\bmzero$}][c]{$\bmzero$} \\
\bmoner & \bmone & \bmone & \bmzero & \bmzero \\
\bmzero & \bmoner & \bmone & \bmone & \bmone\\
\bmzero & \bmoner & \bmone & \bmzero & \bmzero \\
\bmzero & \bmoner & \bmzero & \bmzero & \bmzero \\
\bmone & \bmzero & \bmoner & \bmone & \bmone \\
\bmone & \bmzero & \bmzero & \bmoner & \bmone
\end{array}$}&&\to\scalebox{0.6}{$\begin{array}{ccccc}
\bmoner & \bmzero & \bmzero & \bmzero & \makebox[\widthof{$\bmzero$}][c]{$\bmzero$} \\
\bmzero & \bmoner & \bmone & \bmone & \bmzero \\
\bmzero & \bmoner & \bmone & \bmone & \bmone\\
\bmzero & \bmoner & \bmone & \bmzero & \bmzero \\
\bmzero & \bmoner & \bmzero & \bmzero & \bmzero \\
\bmone & \bmzero & \bmoner & \bmone & \bmone \\
\bmone & \bmzero & \bmzero & \bmoner & \bmone
\end{array}$} \makebox[\widthof{$\to$}][c]{}
\end{alignat*}

Thus, the leftmost equitable partition of $\uu$ is $\uu^*=1|3421|4|3|\cdot$.
\label{ex:rightmost_example2}
\end{example}

As in~\Cref{sec:rightmost}, we begin with a lemma about column $0$.

\begin{lemma} \label{lem:less-than-eq}
While running~\Cref{map:rightmost}, column 0 will always be less than equitably filled or equitably filled.
\end{lemma}

\begin{proof}
The set of columns which are less than equitably filled or equitably filled can only decrease over time, since we never make a move which decreases the number of copies of $\bmone$ in a column below the equitable amount.  Thus, if a column is less than equitably filled at a certain point, it has been less than equitably filled since the beginning.

Consider some $j>0$, and suppose that at some point in the execution of~\Cref{map:rightmost}, column $j$ is less than equitably filled.  By what we have just argued, column $j$ has been less than equitably filled from the beginning.  In particular, we have never pushed any letters from the $j$th block to the $(j-1)$st block, and we have never moved any letters into any block to the right of block $j$ (since any such letter starts in block $m-1$ and so would have had to pass through the $j$th block).  So every column to the right of column $j$, including column 0, contains at most the same number of copies of $\bmone$ as column $j$.  It follows that all columns to the right of column $j$---except possibly column 0---are all also less than equitably filled.  Since an equitable filling of column $j$ has more than or equal to an equitable filling of column 0, column 0 is either less than equitably filled or equitably filled.

Thus, at each step of the algorithm, if the filling is not yet equitable, column 0 is either less than equitably filled or equitably filled.
\end{proof}

\begin{theorem}[{\cite[Lemma 8.1]{thomas2014cyclic}}]
For any word $\uu \in \Am$, there is a unique leftmost equitable partition $\rightmost(\uu)$.
\label{lem:rightmost}
\end{theorem}

\begin{proof}
We claim that~\Cref{map:rightmost} constructs the unique leftmost equitable partition.  We first show that~\Cref{map:rightmost} does not attempt any illegal moves, and so returns an equitable partition.  Let $j$ be the leftmost column that is more than equitably filled at some point in the execution of~\Cref{map:rightmost}.  \Cref{map:rightmost} might fail in two ways:
\begin{enumerate}
\item Suppose $j=0$.  Then~\Cref{map:rightmost} would attempt to delete the rightmost letter of $\bb_0$ and append it to the nonexistent block $\bb_{-1}$.  But this case cannot arise by~\Cref{lem:less-than-eq}.
\item Suppose $j=m-1$.  \Cref{map:rightmost} might try to delete the rightmost entry of $\bb_{m-1}$ when $\bb_{m-1}$ is empty, but more than equitably filled.  Since $\bb_{m-1}$ is empty, the number of copies of $\bmone$ in column $m-1$ is at most the number in column $0$, which is less than equitably filled or equitably filled by~\Cref{lem:less-than-eq}.  So this case also cannot arise.
\item Otherwise, suppose $m-1>j>0$.  \Cref{map:rightmost} might try to delete the rightmost entry of an \emph{empty} block $\bb_{j}$ when column $j$ is more than equitably filled.  As in the previous case, since $\bb_{j}$ is empty, the number of copies of $\bmone$ in column $j$ is at most the number in column $j+1$.  Therefore, this case arises only when an equitable filling of column $j$ has strictly fewer copies of $\bmone$ than an equitable filling of column $j+1$.  But the only column with this property is column $0$.
\end{enumerate}%

At each step letters are moved further to the right, so we never obtain the same partition twice.    Since there are only a finite number of partitioned words,~\Cref{map:rightmost} must eventually terminate.

We now show that \Cref{map:rightmost} outputs the unique leftmost equitable partition $\uu^*$ of $\uu$.  For each step of \Cref{map:rightmost}, record the pair $(j,i)$, where $j$ is the leftmost column that is more than equitably filled, and $i$ is the index of the letter $\uu_i$ that we move from $\uu_j^*$
to $\uu_{j-1}^*$.  Consider another equitable partition $\vv^*$ of $\uu$ and find the first recorded pair $(j,i)$ such that $\uu_i$ is in $\vv_j^*$.  If there is no such pair, then each letter of $\uu$ is at least as far to the right in $\vv^*$ as it is in $\uu^*$ and we are done.

Otherwise, write $\uu^\partial$ for
the partition produced by running~\Cref{map:rightmost} up to (but not including) the step that produced the pair $(j,i)$. By the choice of $(j,i)$, it follows that every letter of $\uu$ must be at least as far to the right in $\vv^*$ as in $\uu^\partial$.  The partition $\uu^\partial$ is non-equitable,
and $j$ is its leftmost column that is more than equitably filled.  Every letter that contributes a $\bmone$ to column $j$ for $\uu^\partial$ is at least as far to the right in $\vv^*$ as in $\uu^\partial$ and since
the
final letter in $\uu_j^\partial$ is by assumption also in $\vv^*_j$, none of these  letters can have moved to the right of
$\vv^*_j$.  Thus, $\vv^*$ has at least as many copies of $\bmone$ in
column $j$ as $\uu^\partial$ does, which is too many to be equitable, contradicting our assumption
that $\vv^*$ was equitable.
\end{proof}

\subsection{The Distributive Lattice and Stable Marriages}

\begin{lemma}\label{lem:covers}
Let $\us^*$ be a minimal left balanced block suffix of $\uu^*$, and let $\vv^*$ be the result of shifting the letters in $\us^*$ one block to the right in $\uu^*$.  Then $\uu^* \lessdot \vv^*$ is a cover relation in the componentwise order on equitable partitions of $\uu$.  All cover relations are of this form.
\end{lemma}

\begin{proof}
%
Suppose that we have a cover relation $\uu^* \lessdot \vv^*$.
Let
$\us^*$ be the letters of $\uu^*$ which have moved to the right in $\vv^*$.  We want to establish
that each of the letters of $\us^*$ has moved exactly one step to the right in $\vv^*$.  But this now follows from~\Cref{lem:cascade}, which establishes that every cover relation is obtained by moving each of the elements of some minimal left balanced block suffix one block to the right.

Conversely, it is clear that moving the elements of a minimal left balanced block suffix of $\uu^*$ one step to the right produces an equitable partition which is higher in the componentwise order than $\uu^*$.   By the previous analysis, this must in fact be a cover relation.
\end{proof}

Given two equitable partitions $\uu^*, \vv^* \in \equ(\uu)$, let $\uu^* \vee \vv^*$ be the partition given by \[\block(\uu^* \vee \vv^*,i):=\min \{\block(\uu^*,i),\block(\vv^*,i)\},\] and let $\uu^* \wedge \vv^*$ be the partition given by \[\block(\uu^* \wedge \vv^*,i):=\max \{\block(\uu^*,i),\block(\vv^*,i)\}.\]

\begin{proposition}
For two equitable partitions $\uu^*,\vv^*\in \equ(\uu)$, both $\uu^* \vee \vv^*$ and $\uu^* \wedge \vv^*$ are equitable partitions.
\end{proposition}

\begin{proof}
We give the argument for $\uu^* \wedge \vv^*$; the dual argument gives the corresponding result for $\uu^*\vee \vv^*$.  The leftmost equitable partition $\rightmost(\uu)$ is certainly a lower bound for both $\uu^*$ and $\vv^*$.  Suppose there exists an element $\uu_i$ at the end of a block in $\rightmost(\uu)$ such that $\block(\rightmost(\uu),i)\neq \max \{\block(\uu^*,i),\block(\vv^*,i)\}$.  Then $\uu_i$ is a member of a minimal left balanced block suffix $\us^*$ of $\rightmost(\uu)$, and by~\Cref{lem:covers} we may move all of the letters in $\us^*$ one block to the right to move up in the componentwise order.  The resulting partition $\mathsf{w}^*$ is still a lower bound for both $\uu^*$ and $\vv^*$, since any letter which we moved to the right must also be at least as far to the right in $\uu^*$ and $\vv^*$.  We may continue this process so long as there is an element $\uu_i$ such that $\block(\mathsf{w}^*,i)\neq \max \{\block(\uu^*,i),\block(\vv^*,i)\}$.  The process therefore terminates with $\uu^* \wedge \vv^*$.
\end{proof}

We conclude the following.

\begin{theorem}
The poset $\equ(\uu)$ is a distributive lattice.  The minimum element is the leftmost equitable partition and the maximum element is the rightmost equitable partition.
\label{thm:dis_lattice}
\end{theorem}

\begin{proof}
Since the operations $\mathrm{max}$ and $\mathrm{min}$ distribute, we obtain the first statement.  The second statement follows from \Cref{lem:rightmost} and \Cref{thm:unique_rightmost}.
\end{proof}



%

\begin{remark}
\Cref{thm:dis_lattice} was found by analogy with stable marriages, as we now explain.  The theory of stable marriages is due to D.~Gale and L.~Shapley~\cite{gale1962college}.  Briefly, given $n$ men and $n$ women who have each individually totally ordered the opposite sex, an \defn{unstable set of marriages} is a bijection between the men and the women such that there exist a man and a woman (not married to each other) who prefer each other to their actual spouses.  A \defn{stable set of marriages} is a bijection between the men and women which is not unstable.  In~~\cite{gale1962college}, D.~Gale and L.~Shapley proved the existence of a set of stable marriages by giving an algorithm that produced one.  This work of D.~Gale and L.~Shapley has since found diverse applications---for example, it provided the mathematical underpinning for reform of residency matches~\cite{williams1981analysis}, such as the San Francisco match for ophthalmology~\cite{williams1996examining,williams1996match} and the National Resident Matching Program (NRMP), which assign graduating medical students to residency programs.

Although the men and women appear to be equals in the statement of the problem, D.~Gale and L.~Shapley's algorithm has the interesting property of not treating the two groups symmetrically.  In their algorithm, the members of one group act by proposing to the members of the other group---working down from the top of their list---while the other group continuously rejects all but the most favorable proposal they have yet received.  This asymmetry is highlighted by the outcome: when the men propose to the women, the output is the worst possible set of stable marriages for the women (in the sense that in any other possible set of stable marriages, every woman would be paired with a man no lower in her total order) and the best possible for the men.  This bias was originally raised in reference to the NRMP in~\cite{williams1981analysis}, where it was shown that the choice had been made to favor the interests of residency programs over those of students.  This bias has now been largely switched, with the consent of all parties; the San Francisco match switched in 1996~\cite{williams1996examining,williams1996match}, while the larger NRMP switched under public pressure two years later.

In~\cite{knuth1976mariages}, D.~Knuth credits J.~Conway with the observation that the set of all stable sets of marriages forms a distributive lattice under componentwise order (D.~Gale and L.~Shapley were already aware of the existence of intermediate stable sets of marriages)---depending on conventions, the minimum element is the best possible stable matching for the men (obtained by running the Gale-Shapley algorithm with men proposing), while the maximum element is the best for the women.  Although we are not aware of an equivalence between our work and the theory of stable marriages, we note that there is a certain similarity between the roughly symmetric~\Cref{map:rightmost,map:leftmost} and the Gale-Shapley proposing algorithms.\footnote{It is not too difficult to imagine a somewhat scandalous rephrasing of~\Cref{sec:real_world} using a polyamorous cul-de-sac.}  Our leftmost equitable partition and rightmost equitable partition are analogous to the male- and female-favoring stable sets of marriages. 
\label{rem:stable_marriages}
\end{remark}

\section{Applications}
\label{sec:applications}

\subsection{The Sweep Map}
\label{sec:general_sweep}

By taking $m$ sufficiently large, the modular sweep map emulates the sweep map of~\cite{armstrong2014sweep}, as we now explain.

Fix $\mathbf{a}:=(a_1,\ldots,a_n) \in \mathbb{Z}^n$, let $\mathbf{e}:=(e_1,\ldots,e_n) \in \mathbb{N}^n$, and define $\AZ$ to be the set of words containing $e_j$ copies of $a_j$ for $1 \leq j \leq n$.  For a word $\w=\w_1 \w_2 \cdots \w_{N} \in \AZ$, define the \defn{level} of $\w_j$ to be the integer $\ell_j:=\sum_{i=1}^j \w_i$ for $1 \leq j \leq N$.

The \defn{sweep map} is the function $\swg:\AZ \to \AZ$ that sorts $\w \in \AZ$ according to its levels as follows: initialize $\uu=\emptyset$ to be the empty word.  For $k=-1,-2,-3,\ldots$ and then $k=\ldots,3,2,1,0$, read $w$ from right to left and append to $\uu$ all letters $\w_j$ whose level $\ell_j$ is equal to $k$.  Define $\swg(\w):=\uu$.

\begin{example}[{\cite[Figure 3]{armstrong2014sweep}}]
Let $n=2$, $(a_1,a_2)=(-2,3)$, and $(e_1,e_2)=(10,8)$.  Writing $\overline{i}:=-i$, we begin with a word $\w \in \AZ$, compute its levels, and find $\uu=\swg(\w)$ as described above:
\[\begin{array}{rcccccccccccccccccc}
\ell:&3&1&4&2&0& \overline{2}& 1&\overline{1}&\overline{3}&\overline{5}&\overline{7}&\overline{4}&\overline{1}&\overline{3}&\overline{5}&\overline{2}&1&4\\
\w:&3&\overline{2}&3&\overline{2}&\overline{2}&\overline{2}&3&\overline{2}&\overline{2}&\overline{2}&\overline{2}&3&3&\overline{2}&\overline{2}&3&3&3\end{array},\]

\[\begin{array}{rcccccccccccccccccccc}
\ell: &\bar{1}&\bar{1}&\bar{2}&\bar{2}&\bar{3}&\bar{3}&\bar{4}&\bar{5}&\bar{5}&\bar{6}&\bar{7}&\cdots&4&4&3&2&1&1&1&0 \\
\uu:&3 &\overline{2} &3& \overline{2}&\overline{2}&\overline{2}&3&\overline{2}&\overline{2}&\cdot&\overline{2}&\cdots&3&3&3&\overline{2}&3&3&\overline{2}&\overline{2}\end{array}.\]

\label{ex:zeta_final}
\end{example}




We may interpret elements of $\AZ$ as paths in the lattice $\mathbb{Z}^n$, with standard basis denoted ${\sf a}_1,{\sf a}_2,\ldots,{\sf a}_n$.   By reading the letter $a_i$ as a step in direction ${\sf dir}(a_i):={\sf a}_i$, a word $\w \in \AZ$ is interpreted as a lattice path $\pat(\w)$ from $(0,0,\ldots,0)$ to $(e_1,e_2,\ldots,e_n)$ using steps ${\sf a}_i$.
The \defn{endpoint} of the step corresponding to $\w_j$ is the point $\sum_{i=1}^j {\sf dir}(\w_i)$.

\begin{example}
Continuing~\Cref{ex:zeta_final}, we may draw $\w$ and $\uu$ as lattice paths from $(0,0)$ to $(10,8)$, where the number $3$ in $\w$ or $\uu$ represents a step north, while a $\bar{2}$ represents a step east.  The endpoints are indicated by black circles, with the corresponding levels from~\Cref{ex:zeta_final} marked inside.

\[
\raisebox{-0.5\height}{\includegraphics[width=.45\linewidth]{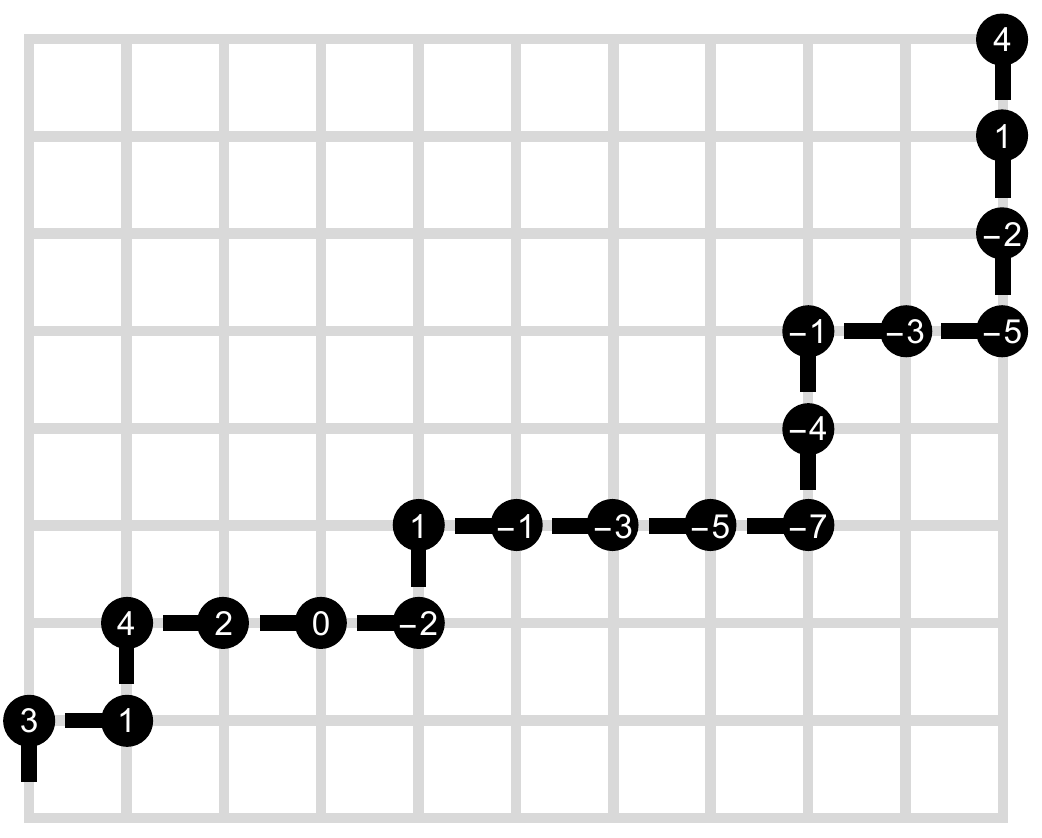}} \xmapsto[\swg]{} \raisebox{-0.5\height}{\includegraphics[width=.45\linewidth]{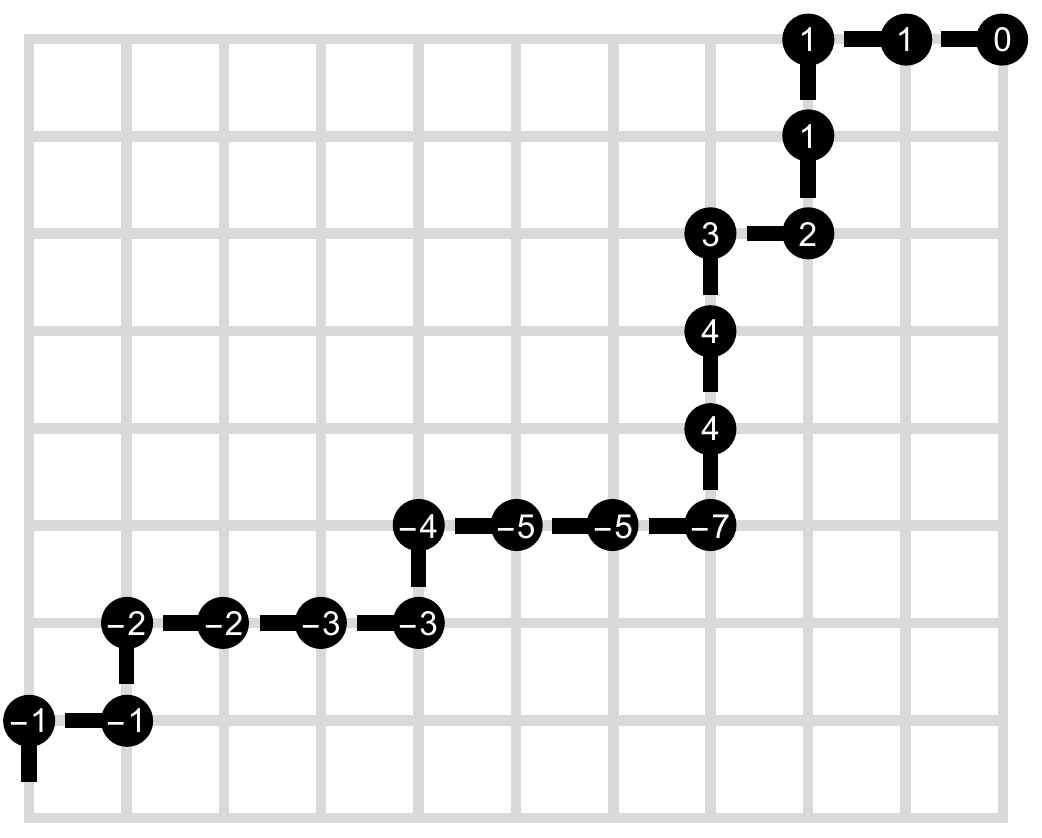}}
\]
\label{ex:zeta_final2}
\end{example}

  One can now visualize the sweep map as the process of \emph{sweeping} the affine hyperplane defined by ${\mathcal H}_{\mathbf{a},k}:=\{\mathbf{x}:\mathbf{x} \bullet \mathbf{a}=k\}$ first down from $k=-1$ to $k=-\infty$, and then down from  $k=\infty$ to $k=0$---recording the letters corresponding to the endpoints of the steps of $\pat(\w)$ in the order they are intersected by this hyperplane.  \Cref{fig:ratpaths} illustrates this geometric interpretation for $n=2$ with $\mathbf{a}=(-4,7)$ and $\mathbf{e}=(7,4)$.

\begin{figure}[htbp]
\begin{alignat*}{4}
\framebox{\raisebox{-0.5\height}{\includegraphics[width=.13\linewidth]{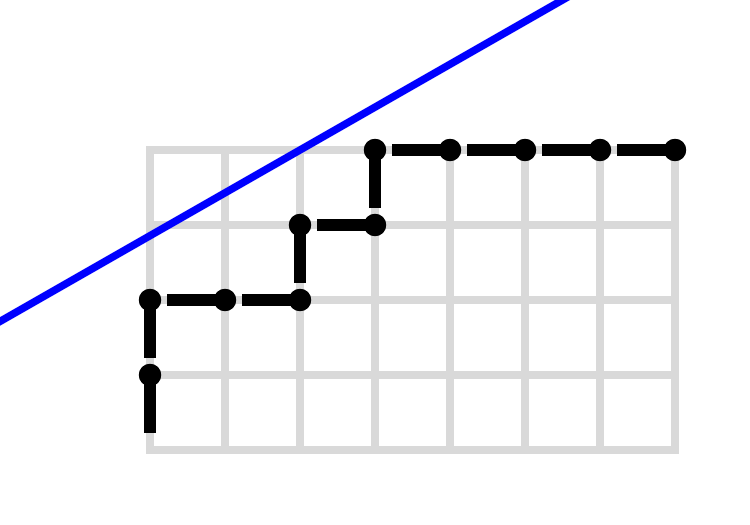}\includegraphics[width=.13\linewidth]{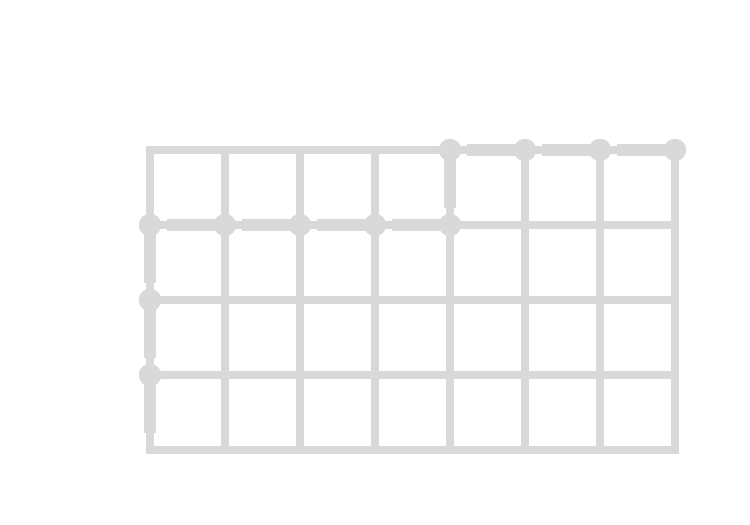}}} &&\to
\framebox{\raisebox{-0.5\height}{\includegraphics[width=.13\linewidth]{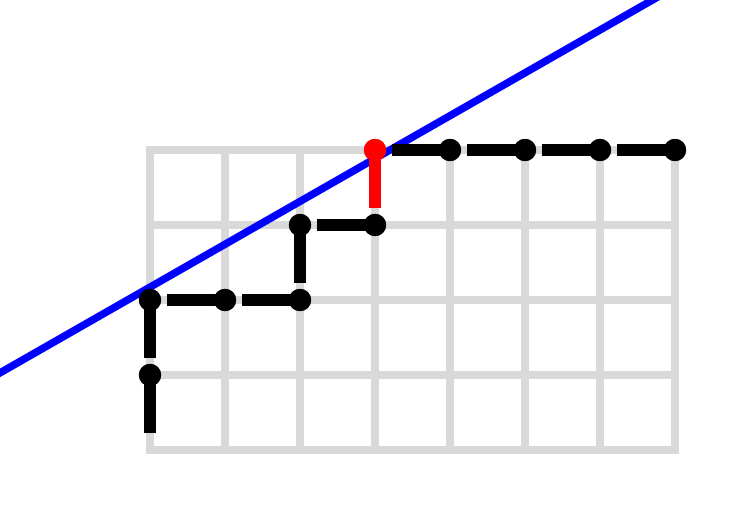}\includegraphics[width=.13\linewidth]{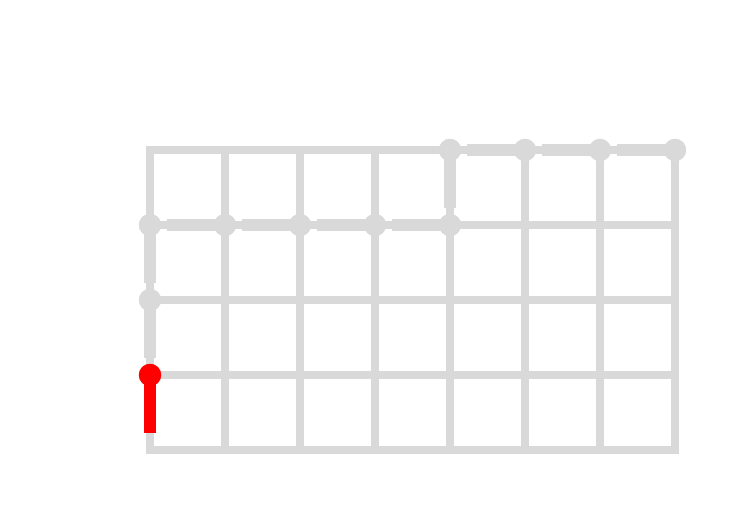}}} &&\to
\framebox{\raisebox{-0.5\height}{\includegraphics[width=.13\linewidth]{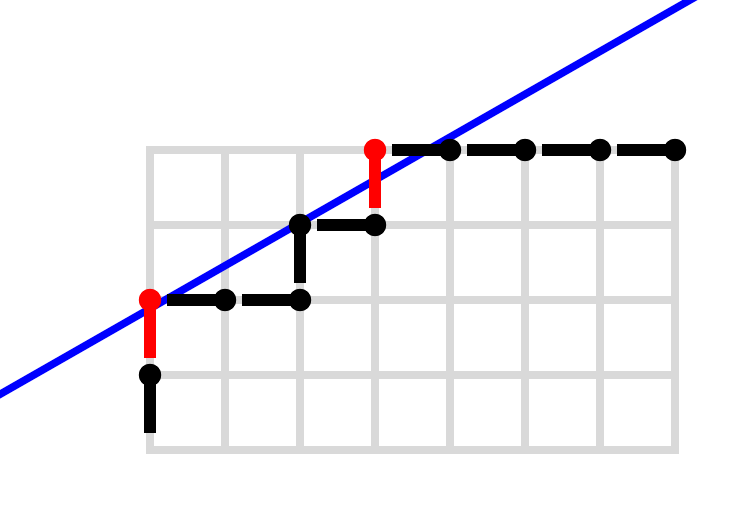}\includegraphics[width=.13\linewidth]{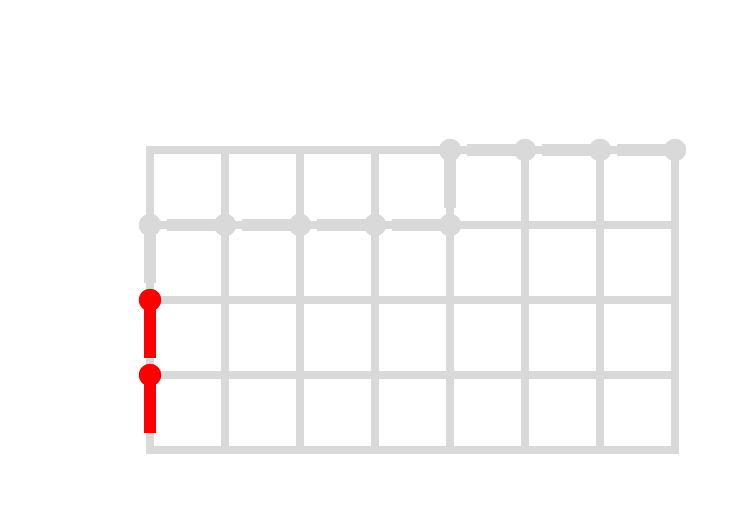}}} &&\to \\
\framebox{\raisebox{-0.5\height}{\includegraphics[width=.13\linewidth]{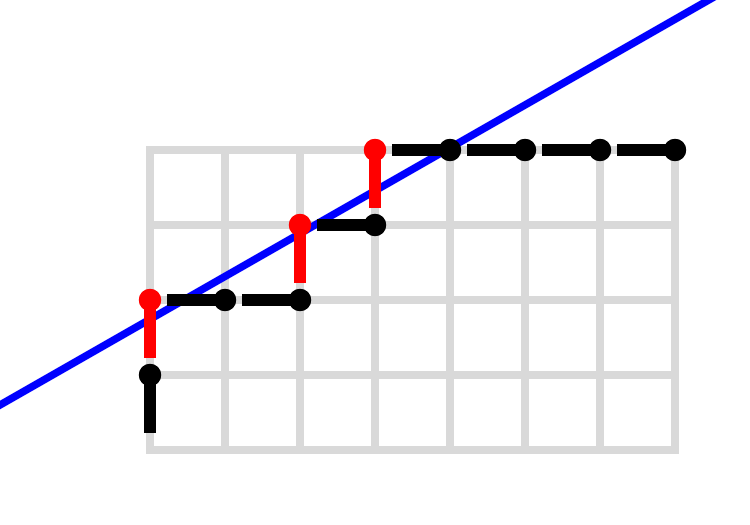}\includegraphics[width=.13\linewidth]{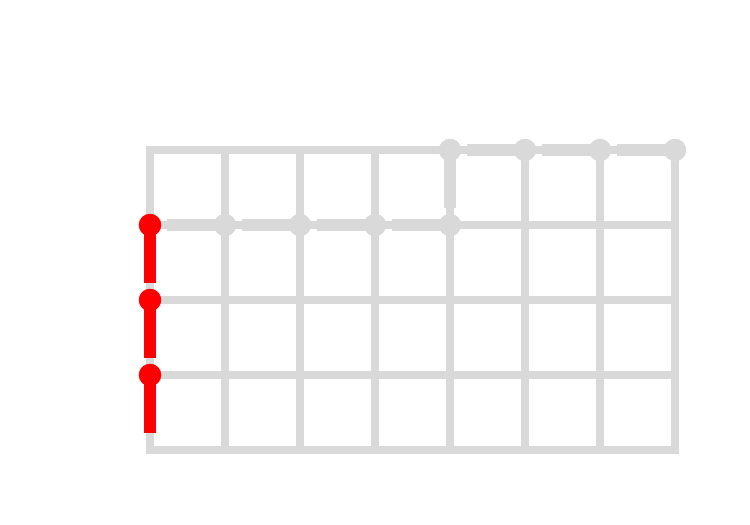}}} &&\to
\framebox{\raisebox{-0.5\height}{\includegraphics[width=.13\linewidth]{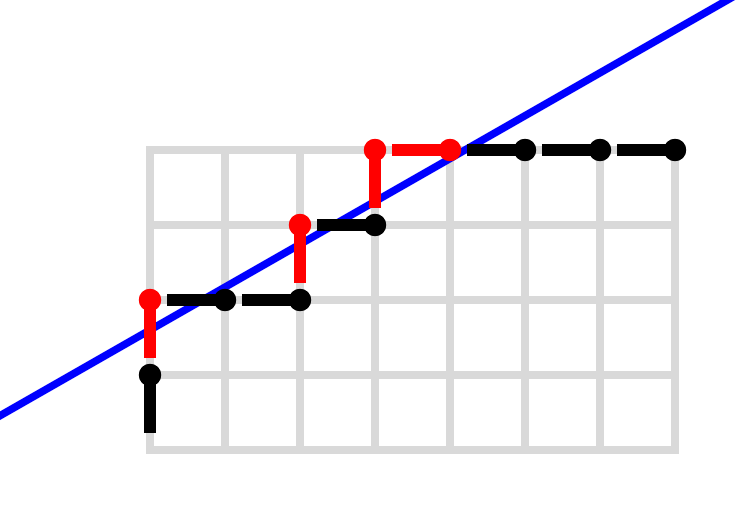}\includegraphics[width=.13\linewidth]{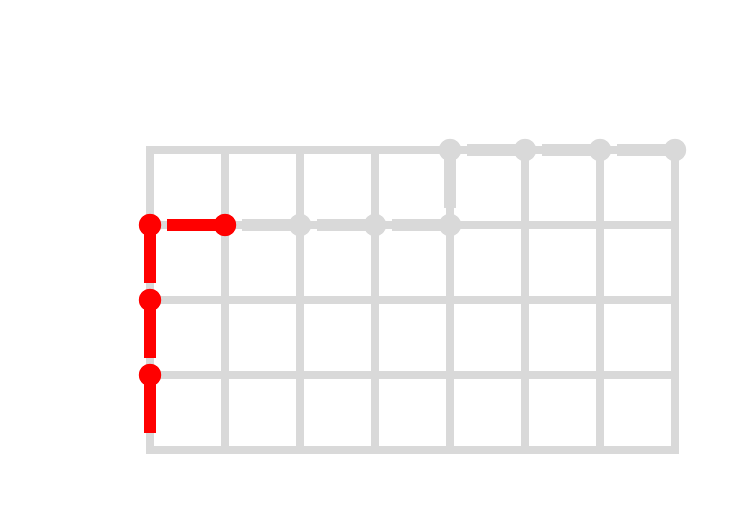}}} &&\to
\framebox{\raisebox{-0.5\height}{\includegraphics[width=.13\linewidth]{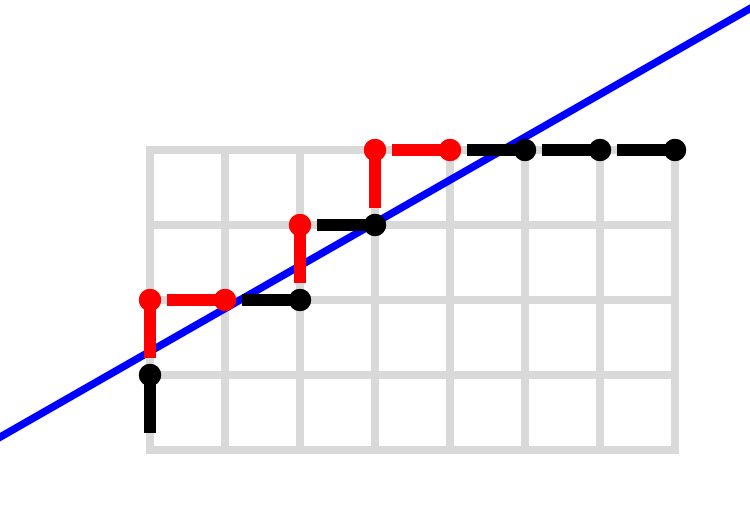}\includegraphics[width=.13\linewidth]{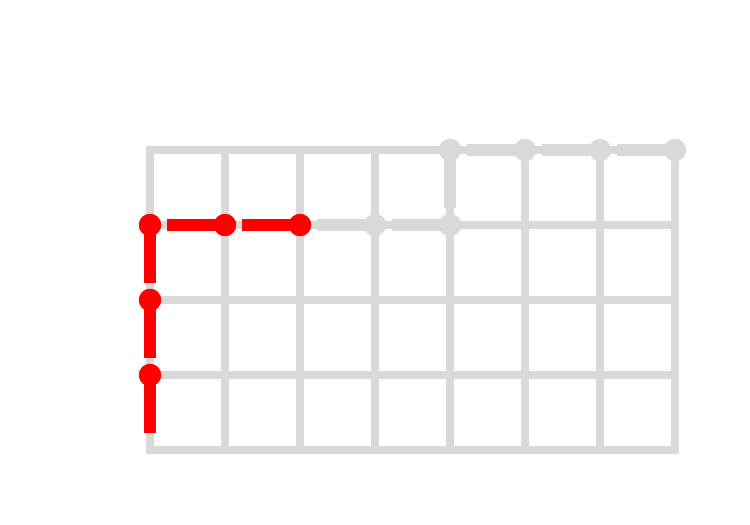}}} &&\to \\
\framebox{\raisebox{-0.5\height}{\includegraphics[width=.13\linewidth]{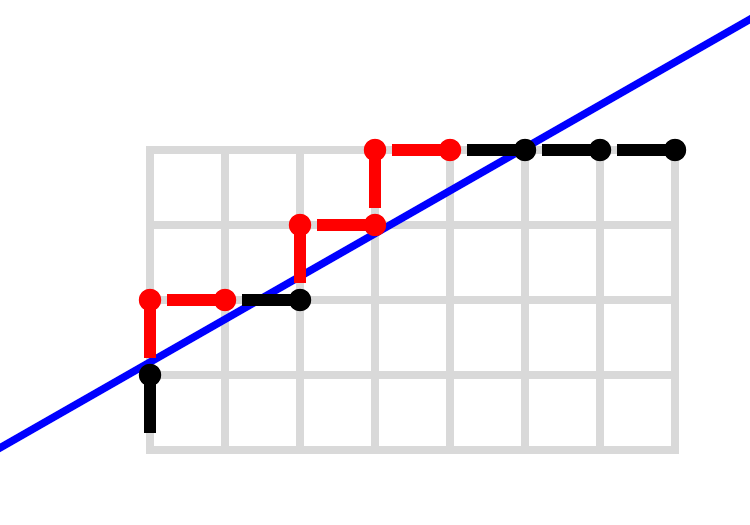}\includegraphics[width=.13\linewidth]{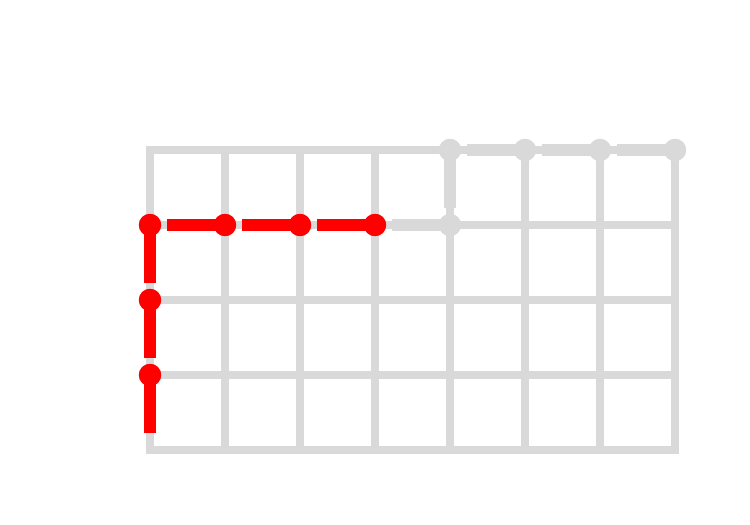}}} &&\to
\framebox{\raisebox{-0.5\height}{\includegraphics[width=.13\linewidth]{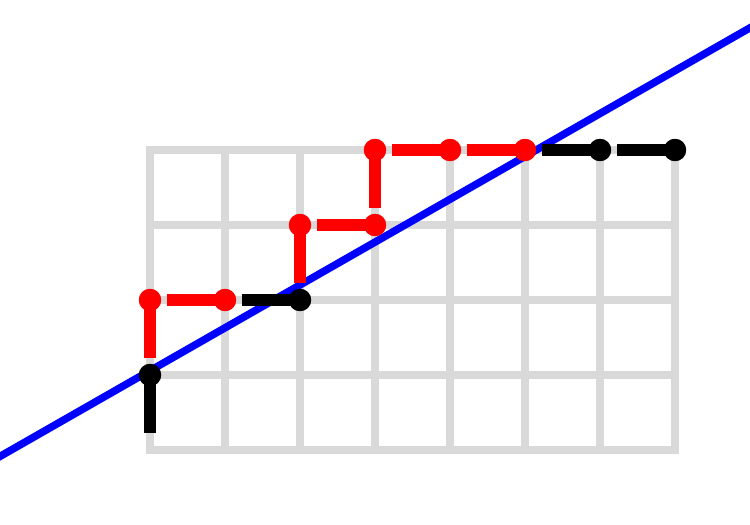}\includegraphics[width=.13\linewidth]{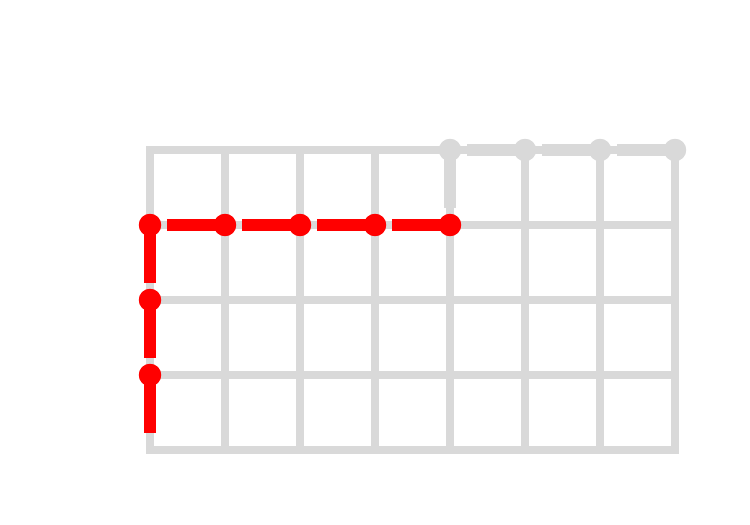}}} &&\to
\framebox{\raisebox{-0.5\height}{\includegraphics[width=.13\linewidth]{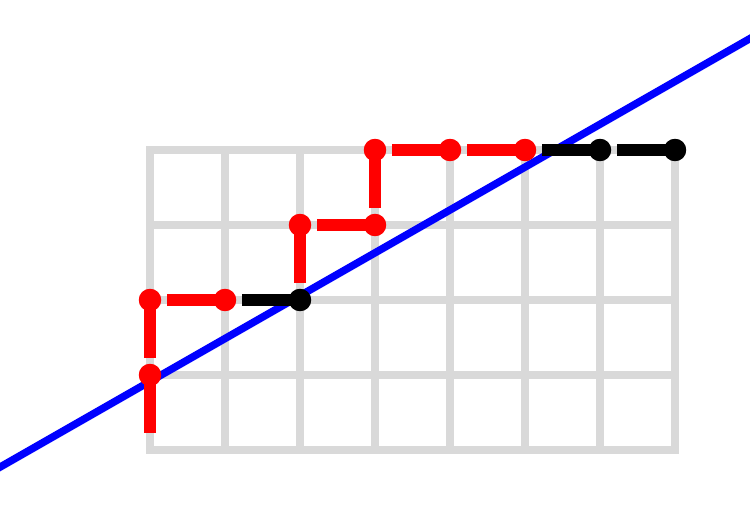}\includegraphics[width=.13\linewidth]{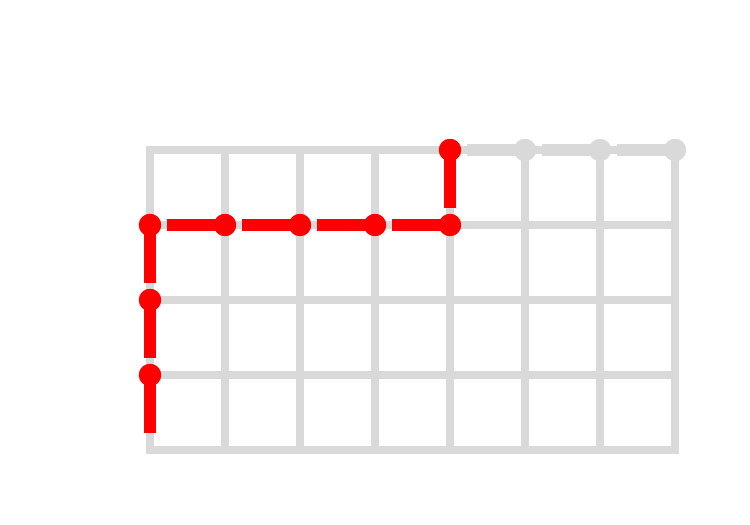}}} &&\to \\
\framebox{\raisebox{-0.5\height}{\includegraphics[width=.13\linewidth]{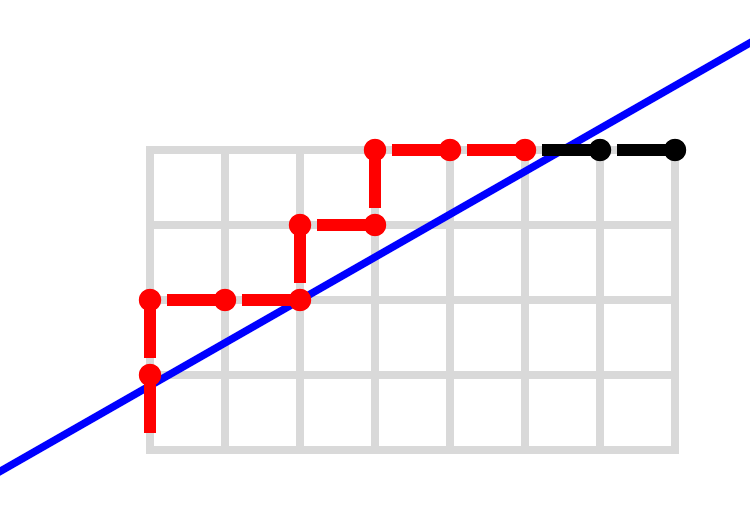}\includegraphics[width=.13\linewidth]{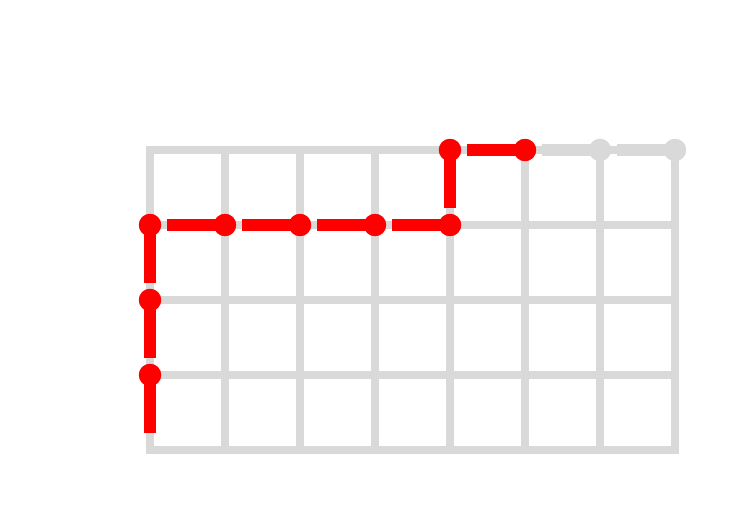}}} &&\to
\framebox{\raisebox{-0.5\height}{\includegraphics[width=.13\linewidth]{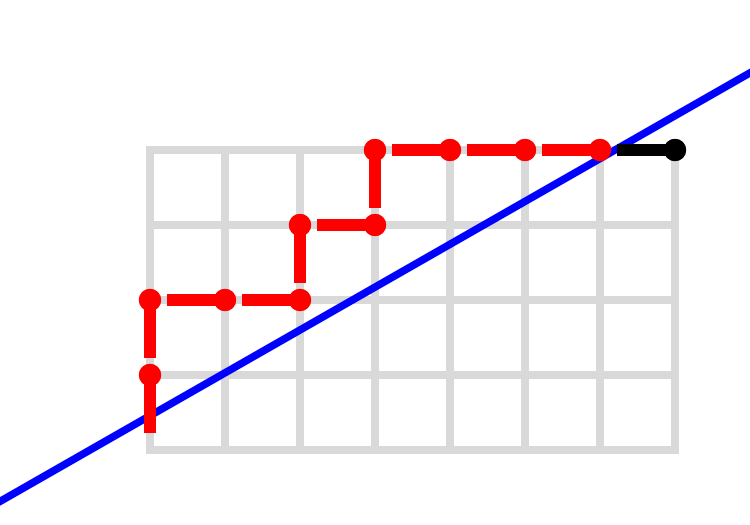}\includegraphics[width=.13\linewidth]{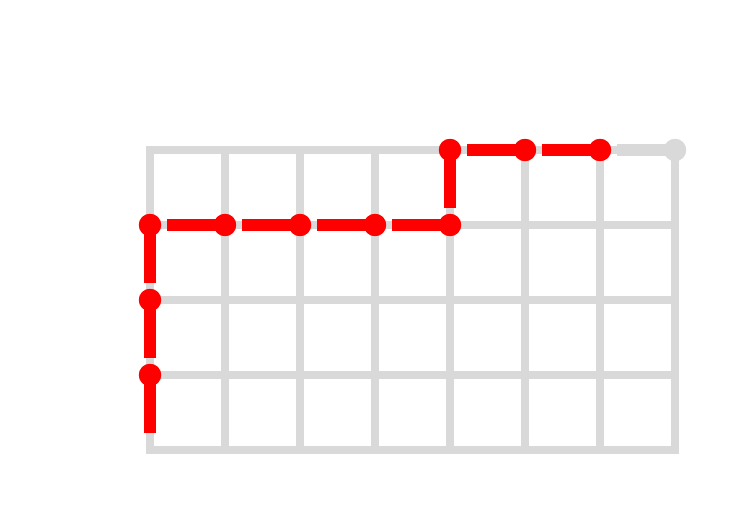}}} &&\to
\framebox{\raisebox{-0.5\height}{\includegraphics[width=.13\linewidth]{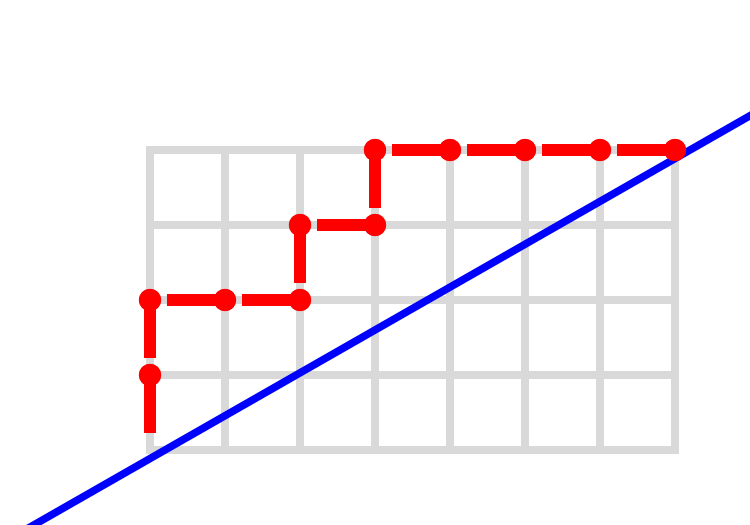}\includegraphics[width=.13\linewidth]{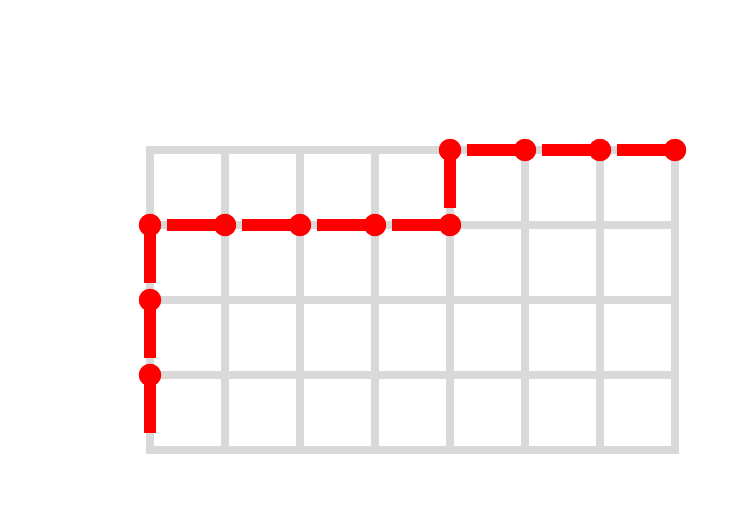}}} &&\makebox[\widthof{$\to$}][c]{$.$}
\end{alignat*}
\caption{An illustration of the geometric interpretation of $\swg$.  To form the right path, the steps of the left path are rearranged according to the order in which they are encountered by a line of slope $4/7$ sweeping down from above.}
\label{fig:ratpaths}
\end{figure}

This definition is slightly misleading---in order to have agreement with $\swg$, one must respect the tiebreaking ``right to left'' condition when two points are intersected by ${\mathcal H}_{\mathbf{a},k}$ \emph{at the same time}.  It is therefore nicer in this geometric interpretation to break these ties by slightly perturbing the vector $\mathbf{a}$, so that we allow $k$ to vary \emph{continuously} first from $k=-1$ to $k=-\infty$ and then from $k=\infty$ to $k=0$ (and still record the letters corresponding to the endpoints of $\pat(\w)$ in the order they are intersected by ${\mathcal H}_{\mathbf{a},k}$).  To break ties, we perturb $\mathbf{a}=(a_1,\ldots,a_n)$ by choosing $\mathbf{a}'=(a_1',a_2',\ldots,a_n')$ linearly independent over $\mathbb{Q}$ such that $a_i'>a_i$ for all $i$, and such that there is no lattice point $\mathbf{x}$  in the hypercube spanned by $\{e_i {\sf a}_i\}_{i=1}^n$ such that $\mathbf{x}\bullet \mathbf{a}<0$ and $\mathbf{x}\bullet \mathbf{a}'>0$.


\begin{theorem}[{\cite[Conjecture 3.3 (a)]{armstrong2014sweep}}]
\label{thm:sweep_bij}
The sweep map is a bijection \[\swg: \AZ \to \AZ.\]
\end{theorem}
\begin{proof}
Since the modular sweep map only \emph{permutes} its input $\w$, it restricts to a bijection on words with a specified content.  We claim that by choosing $m$ large enough, the modular sweep map agrees with the sweep map when the letters $a_j$ and levels $\ell_j$ are taken modulo $m$.

When computing the modular sweep map, so long as $m$ is large enough that
\begin{enumerate}
    \item all distinct letters correspond to distinct letters modulo $m$,
    \item all distinct levels correspond to distinct modular levels, and
    \item modulo $m$, all negative levels are larger than all positive levels,
\end{enumerate}
    then by the definitions of $\swg$ and $\sw$, the images of the sweep and modular sweep maps will be constructed by visiting the same letters in the same order---and so will agree modulo $m$.  Increasing $m$ further does not result in a violation of these conditions, and we conclude that the modular sweep map ``stabilizes'' after the point that it begins emulating the sweep map.  By ``stabilize,'' we mean that that the only change in the output of $\sw$ when the modulus changes from $m$ to $m+1$ is the addition of a single empty new block between the negative and positive levels modulo $m$.

    We can achieve these conditions simultaneously for all words in $\AZ$ by taking $m>\sum_{j=1}^{n} e_j|a_j|$.  Note that this bound does not depend on the levels in the original word---only on the content, which is unchanged under the sweep or modular sweep map.  We conclude that by taking $m$ large enough and the letters $a_i$ modulo $m$, we may use the inverse modular sweep map to compute the inverse sweep map.
\end{proof}
\begin{example}
Continuing~\Cref{ex:zeta_final,ex:zeta_final2}, we verify the three conditions in~\Cref{thm:sweep_bij} for $\sw$ to emulate $\swg$ are satisfied when $m \geq 12$ (this choice of $m$ is only valid for this particular word, not for all rearrangements).  Fixing $m\geq 12$ and writing $\overline{i}:=-i \mod m$, we take the word in~\Cref{ex:zeta_final} modulo $m$ and use the \emph{modular} sweep map to compute that $\sw(\w \mod m)$ agrees with $\uu \mod m$:
\[\begin{array}{rcccccccccccccccccc}
\ell \mod m:&3&1&4&2&0& \overline{2}& 1&\overline{1}&\overline{3}&\overline{5}&\overline{7}&\overline{4}&\overline{1}&\overline{3}&\overline{5}&\overline{2}&1&4\\
\w \mod m:&3&\overline{2}&3&\overline{2}&\overline{2}&\overline{2}&3&\overline{2}&\overline{2}&\overline{2}&\overline{2}&3&3&\overline{2}&\overline{2}&3&3&3\end{array},\]

\[\begin{array}{rcccccccccccccccccccc}
\ell \mod m: &\bar{1}&\bar{1}&\bar{2}&\bar{2}&\bar{3}&\bar{3}&\bar{4}&\bar{5}&\bar{5}&\bar{6}&\bar{7}&\cdots&4&4&3&2&1&1&1&0 \\
\uu \mod m:&3 &\overline{2} &3& \overline{2}&\overline{2}&\overline{2}&3&\overline{2}&\overline{2}&\cdot&\overline{2}&\cdots&3&3&3&\overline{2}&3&3&\overline{2}&\overline{2}\end{array}.\]

The effect of increasing $m$ will be only to increase the number of empty blocks in the word $\sw(\w \mod m)$ between the levels $\bar{7}$ and $4$.

We now indicate the straightforward method to compute the inverse sweep map starting only with the word $\uu$.  \Cref{thm:sweep_bij} tells us that we can take $m>3 \cdot 8 + 2 \cdot 10=44$.  We may apply $\invf$ with this modulus to recover the missing data of the modular levels, and then run $\swp$ and discard the modulus (restoring the letters to their original values in $\mathbb{Z}$) to recover $\w$.
\end{example}

\begin{remark}
It might seem that the definitions above are a weaker formulation than the one used in~\cite{armstrong2014sweep}, which uses an arbitrary alphabet (not necessarily $\mathbb{Z}$) along with a weight function from the alphabet to $\mathbb{Z}$.  In that language, our formulation has replaced the alphabet by its image under the weight function, so that it might seem that we have restricted ourselves only to the case of injective weight functions.  In fact, the two formulations are equivalent for the following two reasons:
\begin{enumerate}
\item $\invf$ depends only on the weight function, and not the alphabet; and
\item $\iswp$ traverses the letters in a specified order.
\end{enumerate}
We recover the formulation using arbitrary alphabets and weight functions as follows: use $\invf$ on the image of the word by the weight function; then apply $\iswp$, recording the letter from the alphabet (rather than its image under the weight function).
\label{rem:simple}
\end{remark}

\begin{example}
We give an example of~\Cref{rem:simple} using the modular sweep map; the interested reader can adapt this to the sweep map using~\Cref{thm:sweep_bij}.

Fix $m=5$ and consider the word with distinct letters $\uu = a_1 b_3 c_4 d_2 e_1 f_4 g_3$.  Define a weight function that maps a letter to its subscript.  Using this weight function to run $\invf$, as in~\Cref{ex:rightmost_eq2}, results in the partitioned word
$\cdot|a_1b_3|c_4d_2|e_1|f_4g_3$.  Using this partitioned word, the output of $\iswp$ is $d_2 g_3 e_1 f_4 b_3 c_4 a_1$, as in~\Cref{ex:successful_zeta}.  Note that there is, for example, no confusion between the letters $e_1$ and $a_1$, even though these letters have the same weight.
\end{example}

\subsection{Sweeping Dyck Words}
\label{sec:dyck}
Preserving the notation of the previous section, let $\AN \subseteq \AZ$ be the subset of words in $\AZ$ whose levels are all nonnegative.  Following~\cite{armstrong2014sweep}, we call $\AN$ the set of \defn{Dyck words}.

Restricting the interpretation of elements of $\AZ$, an element of $\AN$ may be seen as a lattice path $\pat(\w)$ from $(0,0,\ldots,0)$ to $(e_1,e_2,\ldots,e_n)$ using steps ${\sf a}_i$, such that the vector dot product $\mathbf{x} \bullet \mathbf{a}\geq 0$ for any $\mathbf{x}$ on $\pat(\w)$.

\begin{theorem}[{\cite[Conjecture 3.3 (b)]{armstrong2014sweep}}]
\label{thm:sweep_dyck}
The sweep map is a bijection \[\swg: \AN \to \AN.\]
\end{theorem}


\begin{proof}
The following argument was suggested by M.~Thiel, generalizing~\cite[Proposition 3.2]{armstrong2014sweep}.  (C.~Reutenauer has developed a similar argument~\cite{reut2013}.)

Let $\uu=\swg(\w_1\w_2\cdots\w_N)$.  We show that for any $j$, the initial segment $\uu_1\uu_2\cdots \uu_j$ ends on or above $\mathcal{H}_{\mathbf{a},0}$, from which we may conclude that $\uu\in \AN$.  Fix $1 \leq j \leq N$.  By definition of $\swg$, for any such $j$, there exists a $k$ for which all letters of $\w$ arranged into this initial segment of $\uu$ have levels greater than or equal to $k$.  The letters in $\w$ that have been arranged to form the initial segment of $\uu$ correspond to those steps in $\pat(\w)$ with endpoints that lie on or above $\mathcal{H}_{\mathbf{a}',k}$.  Define a \defn{connected piece} of these steps of $\pat(\w)$ on or above $\mathcal{H}_{\mathbf{a}',k}$ to be a maximal collection of steps coming from adjacent letters in $\w$.   Translating any connected piece to start at the origin, we see that each connected piece separately satisfies the Dyck word condition to lie on or above $\mathcal{H}_{\mathbf{a}',0}$, so that their rearrangement \emph{ends} on or above $\mathcal{H}_{\mathbf{a}',0}$.  Since we have chosen $\mathbf{a}'$ so that the sweep map adds letters one at a time, this holds for any $j$.
\end{proof}

\medskip

\subsection{The Zeta Map}\label{sec:zeta} Finally, we consider the special case of Dyck words for an alphabet $\{a,b\}$ of size $n=2$, such that $a>0$ and $b<0$ and where the letter $a$ occurs $-b$ times and the letter $b$ occurs $a$ times.  We shall write this set of Dyck words as $\DD_{a,b}$---for $a$ and $b$ relatively prime, these paths are of fundamental importance for the study of rational (type $A$) Catalan combinatorics~\cite{armstrong2013rational,armstrong2014rational,gorsky2016affine,bodnar2015cyclic,ceballos2016combinatorics,sulzgruber2015rational,xin2015efficient,garsia2016inverting}, but we do not need the assumption of relative primality here.

By~\cite[Table 1]{armstrong2014sweep} and~\cite[Theorem 4.8, Lemma 4.10, Theorem 4.12]{armstrong2014sweep}, the \defn{zeta map} may be defined as a variant of the sweep map $\zeta:\DD_{a,b}\to \DD_{a,b}$ that sorts $\w \in \DD_{a,b}$ as follows: initialize $\uu=\emptyset$ to be the empty word.  For $k=0,1,2,\ldots$ and then $k=\ldots,-3,-2,-1$, read $\w$ from left to right and append to $\uu$ all letters $\w_j$ whose level $\ell_j$ is equal to $k$.  Define $\zeta(\w):=\uu$.

We have the following corollary of~\Cref{thm:sweep_bij}, which is of independent interest.

\begin{corollary}[Zeta for Rational Dyck Paths]
\label{cor:zeta_bij}
The zeta map is a bijection \[\zeta: \DD_{a,b} \to \DD_{a,b}.\]
\end{corollary}
\begin{proof}
For $\w=\w_1\w_2\cdots \w_{N}$, let \begin{align*}{\sf rev}(\w)&:=\w_{N}\cdots\w_2\w_1 \hspace{1em}\text{ and } \\ -\w&:=(-\w_1)(-\w_2)\cdots(-\w_{N}).\end{align*}   Then the zeta map may be computed as \[\zeta(\w)=-\left({\sf rev} \circ \swg \circ {\sf rev}\right) (-\w).\]  Since $\swg$ is a bijection, we conclude that $\zeta$ is a bijection.
\end{proof}

\begin{remark}
When $a$ and $b$ are relatively prime, there cannot be any proper balanced block suffixes---the only solutions to $xa-yb=0$ are when $x=kb$ and $y=ka$ for $k \in \mathbb{Z}$, but a proper balanced block suffix has fewer than $-b$ copies of the letter $a$ and fewer than $a$ copies of the letter $b$.  Since the only possible balanced block suffix is the entire word itself, when the modulus $m$ is taken large enough that the modular sweep map emulates the sweep map (as in~\Cref{thm:sweep_bij}), then the distributive lattice of equitable partitions is simply a chain.  In particular, in this very special case, every equitable partition is just a cyclic shift of the blocks of the successful partition. 
\end{remark}

\bibliographystyle{amsalpha}
\bibliography{zeta}

\appendix

\section{Differences with~\cite{thomas2014cyclic}}\label{sec:diffs}
Although Algorithms 1 and 2 in~\cite{thomas2014cyclic} are \emph{very slightly} different from~\Cref{map:forward,map:backwards} in this paper, the notion of a successful partition here and in~\cite{thomas2014cyclic} agree, as we now explain.

There are three differences between~\Cref{map:backwards} here and Algorithm 2 in~\cite{thomas2014cyclic}.  The first difference is the letters recorded---at the $i$th step, rather than recording the first letter of the current block $\ell_{N-i+1}$ (as we do here), in~\cite{thomas2014cyclic} we instead record the label of the next block $\ell_{N-i}$.

The second and third differences are purely notational.  We labeled blocks in~\cite[Definition 7.3]{thomas2014cyclic} as \[\bb_{\ell_{N}+1}| \bb_{\ell_{N}+2}|\cdots|\bb_{\ell_{N}},\] while here (for consistency with~\cite{armstrong2014sweep}) we label them as \[\bb_{m-1}|\bb_{m-2}|\cdots|\bb_0.\]  \Cref{map:backwards} here compensates for this by beginning with block $\bb_{\ell_{N}}$ (as opposed to $\bb_0$) and subtracting (as opposed to adding) the first letter of $\bb_{\ell_{N-i+1}}$ when computing $\ell_{N-i}$.

Finally, what we call the ``rightmost equitable partition'' in~\cite{thomas2014cyclic} has now become the ``leftmost equitable partition.''  In~\cite{thomas2014cyclic}, we chose the terminology ``rightmost'' to represent the positions of the \emph{block dividers} in the word $\uu$, rather than the positions of the letters of $\uu$ in the blocks---we have changed this convention here to be more intuitive.

\end{document}